\documentclass{article}

\usepackage{PRIMEarxiv}
\usepackage{amsmath}
\usepackage{amsthm}
\usepackage[utf8]{inputenc} 
\usepackage[T1]{fontenc}    
\usepackage{hyperref}       
\usepackage{url}     
\usepackage{algorithm,algorithmicx,algpseudocode}
\usepackage{mathrsfs}
\usepackage{subfigure}
\usepackage{enumerate}
\usepackage{bbm}
\usepackage{float}
\usepackage{tikz}
\usepackage{booktabs}       
\usepackage{amsfonts}       
\usepackage{nicefrac}       
\usepackage{microtype}      
\usepackage{lipsum}
\usepackage{fancyhdr}       
\usepackage{comment}
\graphicspath{{media/}}     

\newcommand{\norm}[1]{\left\Vert#1\right\Vert}
\newcommand{\abs}[1]{\left\vert#1\right\vert}
\newcommand{\inp}[1]{\left\langle#1\right\rangle}

\newcommand{\argmin}{\operatornamewithlimits{arg\,min}}

\def\d{\mathrm{d}}
\def\proj{\mathrm{Proj}}

\def\mid{\mathrm{mid}}

\theoremstyle{plain}
\newtheorem{theorem}{Theorem}[section]
\newtheorem{lemma}[theorem]{Lemma}

\newtheorem{proposition}[theorem]{Proposition}

\theoremstyle{definition}
\newtheorem{definition}[theorem]{Definition}

\theoremstyle{remark}
\newtheorem{remark}{Remark}

\title{A Quasi-Newton Subspace Trust Region Algorithm for Nonmonotone Variational Inequalities in Adversarial Learning over Box Constraints
}

\author{
  Zicheng Qiu\footnotemark[1] \\
   \And
  Jie Jiang\footnotemark[2]\; \footnotemark[4] \\
  \And 
  Xiaojun Chen\footnotemark[3]\\
}

\begin{document}
\maketitle

\footnotetext[1]{Department of Applied Mathematics, The Hong Kong Polytechnic University, Hong Kong, China. Email: \texttt{zi-cheng.qiu@polyu.hk} }

\footnotetext[2]{CAS AMSS-PolyU Joint Laboratory of Applied Mathematics, The Hong Kong Polytechnic University, Hong Kong, China. }
\footnotetext[4]{College of Mathematics and Statistics, Chongqing University, Chongqing, China. Email: \texttt{jiangjiecq@163.com}}

\footnotetext[3]{Department of Applied Mathematics, The Hong Kong Polytechnic University, Hong Kong, China. Email: \texttt{xiaojun.chen@polyu.edu.hk} }

\begin{abstract}
The first-order optimality condition of convexly constrained nonconvex nonconcave min-max optimization problems with box constraints formulates a nonmonotone variational inequality (VI), which is equivalent to a system of nonsmooth equations. In this paper, we propose a quasi-Newton subspace trust region (QNSTR) algorithm for the least squares problems defined by the smoothing approximation of nonsmooth equations. Based on the structure of the nonmonotone VI, we use an adaptive quasi-Newton formula to approximate the Hessian matrix and solve a low-dimensional strongly convex quadratic program with ellipse constraints in a subspace at each step of the QNSTR algorithm efficiently.  We prove the global convergence of the QNSTR algorithm to an $\epsilon$-first-order stationary point of the min-max optimization problem.
Moreover, we present numerical results based on the QNSTR algorithm with different subspaces for a mixed generative adversarial networks in eye image segmentation using real data to show the efficiency and effectiveness of the QNSTR algorithm for solving large-scale min-max optimization problems.
\end{abstract}


\section{Introduction}

Min-max optimization problems have wild applications
 in  games \cite{myerson1991game}, distributional robustness optimization \cite{mohajerin2018data}, robust machine learning \cite{madry2017towards}, generative adversarial networks (GANs) \cite{goodfellow2014generative}, reinforcement learning \cite{dai2018sbeed}, distributed optimization \cite{rabbat2004distributed}, etc. Mathematically, a convexly constrained min-max optimization problem can be written as
\begin{equation}
\label{MINIMAX}
    \min_{x\in X}\max_{y\in Y}~ f(x,y):=\mathbb{E}_{P}\left[\ell(x,y,\xi)\right],
\end{equation}
where $X\subseteq \mathbb{R}^{n}$ and $Y\subseteq \mathbb{R}^{m}$ are nonempty, closed and convex sets, $\xi$ is an $s$-dimensional random vector obeying a probability distribution $P$ with support set $\Xi$, $\ell:\mathbb{R}^{n}\times\mathbb{R}^{m}\times\mathbb{R}^{s}\rightarrow\mathbb{R}$ is nonconvex-nonconcave for a fixed realization of $\xi$, i.e., $\ell(x,y,\xi)$ is neither convex in $x$ nor concave in $y$. Hence the objective function  $f(x,y)$ is also nonconvex-nonconcave in general.

In this paper, we are mainly interested in problem (\ref{MINIMAX}) arising from GANs \cite{goodfellow2014generative}, which reads
\begin{equation}\label{GAN}
\begin{aligned}
\min_{x\in X}\max_{y\in Y} \Big( \mathbb{E}_{P_1}\big[\log (D(y,\xi_1))\big] + \mathbb{E}_{P_2}\big[ \log (1- D(y,G(x,\xi_2))) \big] \Big),
\end{aligned}
\end{equation}
where $\xi_i$ is an $\mathbb{R}^{s_i}$-valued random vector with probability distribution $P_i$ for $i=1,2$, $D:\mathbb{R}^m \times \mathbb{R}^{s_1} \to (0,1)$ is a discriminator, $G:\mathbb{R}^n \times \mathbb{R}^{s_2} \to \mathbb{R}^{s_1}$ is a generator.

Generally, the generator $G$ and the discriminator $D$ are two feedforward neural networks. For example, $G$ can be a $p$-layer neural network and $D$ can be a $q$-layer neural network, that is
\begin{align*}
G(x,\xi_2)&= \sigma_G^p (W^{p}_{G}\sigma_G^{p-1}(\cdots\sigma_G^1(W_G^{1}\xi_{2}+b^{1}_{G})+\cdots)+b^{p}_{G}), \\
D(y,\xi_1)&= \sigma_D^q(W_D^q\sigma_D^{q-1}(\cdots\sigma_D^1(W_D^1 \xi_1 + b_D^1)+\cdots)+b_D^q),
\end{align*}
where $W_G^1,\cdots,W_G^p$, $b_G^1,\cdots,b_G^p$ and $W_D^1,\cdots,W_D^q$, $b_D^1,\cdots,b_D^q$ are the weight matrices, biases vectors of $G$ and $D$ with  suitable dimensions, $\sigma_G^1,\cdots,\sigma_G^p$ and $\sigma_D^1,\cdots,\sigma_D^q$ are proper activation functions, such as ReLU, GELU, Sigmoid, etc. Denote
\begin{align*}
x&:=(\mathrm{vec}(W^{1}_{G})^\top,\cdots,\mathrm{vec}(W^{p}_{G})^\top,(b^{1}_{G})^\top,\cdots,(b^{p}_{G})^\top)^\top,\\ y&:=(\mathrm{vec}(W^{1}_{D})^\top,\cdots,\mathrm{vec}(W^{q}_{D})^\top,(b^{1}_{D})^\top,\cdots,(b^{q}_{D})^\top)^\top,
\end{align*}
where $\mathrm{vec}(\cdot)$ denotes the vectorization operator. In this case, problem \eqref{GAN} reduces to problem  \eqref{MINIMAX} if let $\xi:=(\xi_1,\xi_2)\in\Xi$ and
$$\ell(x,y,\xi) := \log (D(y,\xi_1)) +  \log (1- D(y,G(x,\xi_2))).$$

Due to the nonconvexity-nonconcavity of the objective function $f$, problem \eqref{MINIMAX} may not have a saddle point. Hence the concepts of global and local saddle points are untimely to characterize the optimum of problem \eqref{MINIMAX}.  Recently, motivated by practical applications, the so-called global and local minimax points are proposed to describe the global and local optima of nonconvex-nonconcave min-max optimization problems in \cite{jin2020local} from the viewpoint of sequential games. Moreover, the optimality necessary conditions for a local minimax point are studied in \cite{jin2020local} for unconstrained  min-max optimization problems.  In \cite{dai2020optimality,jiang2023optimality},  the optimality conditions for a local minimax point are studied for constrained  min-max optimization problems.

Numerical methods for min-max optimization problems have been extensively studied. These algorithms can be divided into four classes based on the convexity or concavity of problems: the convex-concave cases (see, e.g., \cite{nemirovski2004prox,nesterov2007dual,monteiro2011complexity,tseng2008accelerated}), the nonconvex-concave cases (see, e.g., \cite{rafique2022weakly,lin2020gradient,lin2020near}), the convex-nonconcave cases (see, e.g., \cite{rafique2022weakly,lin2020gradient,lin2020near}) and the nonconvex-nonconcave cases (see, e.g., \cite{diakonikolas2021efficient,yang2020global}).

To solve problem \eqref{MINIMAX} numerically,  we first apply the sample average approximation (SAA) approach
to obtain a discrete form. We collect $N$ independent identically distributed (i.i.d.) samples of $\xi$ (e.g., generated by the Monte Carlo method), denoted by $\xi^1,\cdots,\xi^N$, and obtain a discrete counterpart of problem \eqref{MINIMAX} as below:
\begin{equation}
\label{SAA_MINIMAX}
\min_{x\in X}\max_{y\in Y} ~\hat{f}_{N}(x,y):=\frac{1}{N}\sum_{i=1}^{N}\ell(x,y,\xi^{i}).
\end{equation}

To ease the discussion, we assume that $\ell(\cdot,\cdot,\xi)$  is twice continuously differentiable with respect to $(x,y)$ for an arbitrary $\xi\in \Xi$ in what follows.

Let $z:=(x^\top, y^\top)^\top \in \mathbb{R}^{n+m}$, $Z:=X\times Y\subseteq \mathbb{R}^{n+m}$ and
\begin{equation*}
H_N(z):=
\begin{pmatrix}
\nabla_x \hat{f}_N(x,y)\\
-\nabla_y \hat{f}_N(x,y)
\end{pmatrix}.
\end{equation*}

Then the first-order optimality condition of a local minimax point for problem (\ref{SAA_MINIMAX}) can be presented as the following nonmonotone  variational inequality (VI):
\begin{equation}\label{gs10}
0\in H_N(z)+\mathcal{N}_{Z}(z),
\end{equation}
where $\mathcal{N}_{Z}(z)$ is the normal cone to the convex set $Z$ at $z$, which is defined by 
$$\mathcal{N}_Z(z):= \{v: \inp{v, u -z}\leq 0, \forall u\in Z\}.$$
We call $z^*$ a
first-order stationary point of problem (\ref{SAA_MINIMAX}) if it satisfies
(\ref{gs10}).
Due to the nonconvexity-nonconcavity, seldom algorithms can ensure the convergence to a global or local optimal point of problem \eqref{SAA_MINIMAX}. In the study of the convergence of algorithms for nonconvex-nonconcave min-max problem \eqref{SAA_MINIMAX}, some strong assumptions, such as the  Polyak-\L ojasiewicz (PL) condition \cite{yang2020global}, the existence of solutions for the corresponding Minty VI of problem \eqref{gs10} \cite{diakonikolas2021efficient} etc, are required. In fact, such assumptions are difficult to check in practice. On the other hand, some algorithms need to estimate the Lipschitz constant of $H_{N}(z)$, but the computation of the Lipschitz constant of $H_{N}(z)$ may be too costly or even intractable. In this paper, without estimating the Lipschitz constant of $H_{N}(z)$ and assuming the PL condition or the existence of solutions for Minty VI, we use a so-called quasi-Newton subspace trust region (QNSTR, for short) algorithm for solving problem \eqref{gs10}.

The VI (\ref{gs10})
can be equivalently reformulated as the following system of nonsmooth equations (see \cite{facchinei2003finite})
\begin{equation}
\label{gs15}
F_N(z):=z - {\rm Proj}_Z(z - H_N(z)) = 0,
\end{equation}
where $\mathrm{Proj}_Z(\cdot)$ denotes the projection operator onto $Z$.

Obviously, $z^*$ is a first-order stationary point of (\ref{SAA_MINIMAX}) (i.e., a solution of \eqref{gs10} or \eqref{gs15}) if  it is an optimal solution of  the following least squares problem:
\begin{equation}
\label{TRM_min}
    \min_{z\in \mathbb{R}^{n+m}} r_N(z):=\frac{1}{2}\Vert F_N(z)\Vert^{2}
\end{equation}
and $r_N(z^*)=0$, where $\|\cdot\|$ denotes the Euclidean norm.

The main contributions of this paper are summarized as follows. (i) We develop the QNSTR algorithm for solving the least squares problem (\ref{TRM_min}) when $X$ and $Y$ are boxes. Based on the structure of the VI (\ref{gs10}), we use a smoothing function to approximate the nonsmooth function $F_N$, adopt an adaptive quasi-Newton formula to approximate the Hessian matrix and solve a quadratic program with ellipse constraints in a subspace at each step of the QNSTR algorithm with a relatively low computational cost. (ii) We prove the global convergence of the QNSTR algorithm to a stationary point of a smoothing least squares problem of (\ref{TRM_min}), which is an $\epsilon$-first-order stationary point  of the min-max optimization problem \eqref{SAA_MINIMAX} if every element of the generalized Jacobian matrix of $F_N$ is nonsingular at it.
(iii) We apply the QNSTR algorithm to GANs in eye image segmentation with real data, which validates the effectiveness and efficiency of our approach for large-scale min-max optimization problems.

This paper is organized as follows. In Section \ref{Sec2}, we introduce concepts of local minimax points and first-order optimality conditions. Moreover, we investigate the asymptotic convergence between problems \eqref{MINIMAX} and \eqref{SAA_MINIMAX} to build the numerical foundation for the subsequent discussion. In Section \ref{Sec3}, we present the QNSTR algorithm and establish its global convergence. In Section \ref{Sec4}, we apply the QNSTR algorithm to solve problem \eqref{SAA_MINIMAX} with examples from eye image segmentation problems and digital handwriting image generation problems based on two different real data sets. We compare the QNSTR algorithm with some existing methods including alternating Adam. Finally, we give some concluding remarks in Section \ref{Sec5}.

\textbf{Notations.} $\mathbb{N}$ denotes the set of  natural numbers. $\Vert\cdot\Vert$ denotes the Euclidean norm of a vector or the norm of a matrix induced by the Euclidean norm. ${\rm d}(x, Y):=\inf_{y\in Y} \Vert x - y\Vert$ and ${\rm d}(X, Y):= \sup_{x\in X}\inf_{y\in Y} \Vert x-y\Vert$,  $X, Y \subseteq \mathbb{R}^{n}$.

\section{First-order stationarity and asymptotic convergence}\label{Sec2}
\label{sec:1}
In this section, we focus on the asymptotic convergence between problems \eqref{SAA_MINIMAX} and \eqref{MINIMAX} regarding to the global minimax point and the first-order stationary point. To this end, we first give some preliminaries on how to describe the optima of a min-max optimization problem.

\begin{definition}[global and local minimax points, {\cite[Definitions 9 \& 14]{jin2020local}}]\label{Def1} \hspace{0.8in}
	\begin{itemize}
		\item[(i)] $(\hat{x},\hat{y})\in X\times Y$ is called a\emph{ global minimax point} of problem \eqref{MINIMAX}, if
		$$f(\hat{x},y) \leq f(\hat{x},\hat{y}) \leq \max_{y'\in Y}f(x,y'),~\forall (x,y)\in X\times Y.$$
		
		\item[(ii)] $(\hat{x},\hat{y})\in X\times Y$ is called a \emph{local minimax point} of problem \eqref{MINIMAX}, if there exist a $\delta_0>0$ and a function $\tau:\mathbb{R}_+\rightarrow \mathbb{R}_+$ satisfying $\tau(\delta)\rightarrow 0$ as $\delta\rightarrow 0$, such that for any $\delta\in (0,\delta_0]$ and any $(x,y)\in X\times Y$ satisfying $\norm{x-\hat{x}}\leq \delta$ and $\norm{y-\hat{y}}\leq \delta$, we have
		$$f(\hat{x},y) \leq f(\hat{x},\hat{y}) \leq \max_{y'\in\{y\in Y: \norm{y-\hat{y}}\leq \tau(\delta)\}}f(x,y').$$
	\end{itemize}
\end{definition}

\begin{remark}
The concept of saddle points has been commonly used to characterize the optima of min-max problems. A point  $(\hat{x},\hat{y})\in X\times Y$ is called a \emph{saddle point} of problem \eqref{MINIMAX}, if
\begin{equation}\label{saddle}
f(\hat{x},y) \leq f(\hat{x},\hat{y}) \leq f(x,\hat{y}),~\forall (x,y)\in X\times Y,
\end{equation}
 and   $(\hat{x},\hat{y})\in X\times Y$ is called a \emph{local saddle point} of problem \eqref{MINIMAX} if (\ref{saddle}) holds in a neighborhood of $(\hat{x},\hat{y})$. However, as pointed out in \cite{jin2020local}, saddle points and local saddle points may not exist in many applications of machine learning, especially in the nonconvex-nonconcave case. Also, (local) saddle points are solutions from the viewpoint of simultaneous game, where the minimization operator and the maximization operator act simultaneously. However, many applications, such as GANs and adversarial training, seek for solutions in the sense of sequential game, where the minimization operator acts first and the maximization operator acts latter. The global and local minimax points exist under some mild conditions (see \cite[Proposition 11 and Lemma 16]{jin2020local}) and also describe the optima in the sense of sequential game.
\end{remark}

The following lemma gives the first-order necessary optimality conditions of local minimax points for problem \eqref{MINIMAX}.

\begin{lemma}[{\cite[Theorem 3.2 \& Corollary 3.1]{jiang2023optimality}}]
\label{Lem1}
If $(\hat{x},\hat{y})\in X\times Y$ is a local minimax point of problem \eqref{MINIMAX}, then we have
		\begin{equation}
		\label{gs1}
		\begin{cases}
		0\in \nabla_x f(\hat{x},\hat{y}) + \mathcal{N}_{X}(\hat{x}),\\
		0\in - \nabla_y f(\hat{x},\hat{y}) + \mathcal{N}_{Y}(\hat{y}).
		\end{cases}
		\end{equation}
\end{lemma}

\begin{definition}\label{Def2}
	$(\hat{x},\hat{y})\in X\times Y$ is called a \emph{first-order stationary point} of problem \eqref{MINIMAX} if \eqref{gs1} holds. $(\hat{x},\hat{y})\in X\times Y$ is called a \emph{first-order stationary point} of problem \eqref{SAA_MINIMAX} if \eqref{gs1} holds with replacing $f$ by $\hat{f}_N$.

\end{definition}

Hereafter, we will focus on finding a first-order stationary point of \eqref{SAA_MINIMAX}.

As for the exponential rate of convergence of the first-order and second-order stationary points of SAA for a specific GAN, one can refer to \cite[Proposition 4.3]{jiang2023optimality}. In what follows, we mainly focus on the almost surely convergence analysis between problems \eqref{MINIMAX} and \eqref{SAA_MINIMAX} as $N$ tends to infinity. If the problem is well-behaved and the global minimax points are achievable, we consider the convergence of global minimax points between problems \eqref{MINIMAX} and \eqref{SAA_MINIMAX}. Otherwise, the first-order stationary points (Definition \ref{Def2}) are getatable. Thus, we need also consider the convergence of first-order stationary points between problems \eqref{MINIMAX} and \eqref{SAA_MINIMAX} as $N$ tends to infinity.

Denote the optimal value, the set of global minimax points and the set of first-order stationary points of problem \eqref{MINIMAX} by $\vartheta_g$, $\mathcal{S}_{g}$ and $\mathcal{S}_{1\mathrm{st}}$, respectively. Let $\widehat{\vartheta}_g^N$, $\widehat{\mathcal{S}}_{g}^{N}$ and $\widehat{\mathcal{S}}_{1\mathrm{st}}^{N}$ denote the optimal value, the set of global minimax points and the set of first-order stationary points of problem \eqref{SAA_MINIMAX}, respectively.

\begin{lemma}\label{Lem2}
	Suppose that: (a) $X$ and $Y$ are compact sets; (b) $\ell(x,y,\xi)$ is dominated by an integrable function for every $(x,y)\in X\times Y$. Then
	\begin{align*}
	\sup_{(x,y)\in X\times Y} \abs{\hat{f}_{N}(x,y) -f(x,y)}\to 0
	\end{align*}
	w.p.1 as $N\to \infty$.
	
	If, further, (c) $\norm{\nabla_x \ell(x,y,\xi)}$ and $\norm{\nabla_y \ell(x,y,\xi) }$ are dominated by an integrable function for every $(x,y)\in X\times Y$, then
	\begin{align*}
	\sup_{(x,y)\in X\times Y} \norm{\nabla \hat{f}_{N}(x,y) - \nabla f(x,y)}\to 0
	\end{align*}
	w.p.1 as $N\to \infty$.

\end{lemma}

\begin{proof}
Since the samples are i.i.d. and $X$ and $Y$ are compact, it is known from \cite[Theorem 7.53]{shapiro2021lectures} that the above uniform convergence results hold.\qed
\end{proof}

The following proposition gives the nonemptiness of $\widehat{\mathcal{S}}_{g}^{N}$, $\mathcal{S}_{g}$, $\widehat{\mathcal{S}}_{\mathrm{1st}}^{N}$ and $\mathcal{S}_{\mathrm{1st}}$.

\begin{proposition}\label{Prop1}
	If conditions (a)-(c) in Lemma \ref{Lem2} hold, then $\mathcal{S}_{g}$ and $\mathcal{S}_{\mathrm{1st}}$ are nonempty and $\widehat{\mathcal{S}}_{g}^{N}$ and $\widehat{\mathcal{S}}_{\mathrm{1st}}^{N}$ are nonempty for any $N\in \mathbb{N}$.
\end{proposition}

\begin{proof}
	Since the continuity of $f(x,y)$ and $\hat{f}_{N}(x,y)$ w.r.t. $(x,y)$ and the boundedness of $X$ and $Y$, we know from \cite[Proposition 11]{jin2020local} the nonemptiness of $\mathcal{S}_{g}$ and $\widehat{\mathcal{S}}_{g}^{N}$. Note that both $\mathcal{S}_{\mathrm{1st}}$ and $\widehat{\mathcal{S}}_{\mathrm{1st}}^{N}$ are solutions of variational inequalities. Then we have from \cite[Corollary 2.2.5]{facchinei2003finite} that $\mathcal{S}_{\mathrm{1st}}$ and $\widehat{\mathcal{S}}_{\mathrm{1st}}^{N}$ are nonempty.\qed
\end{proof}

Based on the uniform laws of large numbers in Lemma \ref{Lem2}, we have the following convergence results.

\begin{theorem}\label{Th3}
	Let conditions (a)-(c) in Lemma \ref{Lem2} hold. Then
	\begin{align}
	&\d\left(\widehat{\mathcal{S}}_{g}^{N}, \mathcal{S}_{g} \right)\to 0, \label{gs3-1}\\
	&\d\left(\widehat{\mathcal{S}}_{1\mathrm{st}}^{N}, \mathcal{S}_{1\mathrm{st}} \right)\to 0,  \label{gs3-2}
	\end{align}
	w.p.1 as $N\to \infty$.
\end{theorem}

\begin{proof}
	\eqref{gs3-2} follows from \cite[Proposition 19]{shapiro2003monte} directly. Thus, in what follows, we only consider \eqref{gs3-1}.

	From Proposition \ref{Prop1}, we know that $\widehat{\mathcal{S}}_{g}^{N}$ and $\mathcal{S}_{g}$ are nonempty for any $N\in\mathbb{N}$. Let $z^{N}=(x^{N}, y^{N}) \in \widehat{\mathcal{S}}_{g}^{N}$ and $z^{N} \to \bar{z}=(\bar{x}, \bar{y})$ w.p.1 as $N\to \infty$. Then we just verify that $\bar{z}\in  \mathcal{S}_{g}$ w.p.1. If $\{z^{N}\}$ is not a convergent sequence, due to the boundedness of $X$ and $Y$, we can choose a convergent subsequence.
	Denote $\varphi(x):=\max_{y\in Y} f(x,y)$ and $\hat{\varphi}_{N}(x):=\max_{y\in Y} \hat{f}_{N}(x,y)$.
	Note that
\begin{equation}
\label{zgs1}
	\begin{aligned}
	\max_{x\in X}\abs{\hat{\varphi}_{N}(x) - \varphi(x)} & = \max_{x\in X}\abs{ \max_{y\in Y} \hat{f}_{N}(x,y) - \max_{y\in Y} f(x,y)}\\
	&\leq \max_{(x,y)\in X\times Y} \abs{\hat{f}_{N}(x,y) - f(x,y) }\\
	& \to 0
	\end{aligned}
\end{equation}
	w.p.1 as $N\to \infty$, where the last convergence assertion follows from Lemma \ref{Lem2}.

Next, we show
\begin{equation}
\label{zgs2}
\proj_x \widehat{\mathcal{S}}_{g}^{N}=\argmin_{x\in X} \hat{\varphi}_{N}(x)~~ \text{and} ~~ \proj_x \mathcal{S}_{g}=\argmin_{x\in X} \varphi(x),
\end{equation}
where $\proj_x$ denotes the projection onto the $x$'s space.  Based on the definition of global minimax points, we have, for any $(\hat{x},\hat{y})\in \mathcal{S}_{g}$, that
$$f(\hat{x},y)\le f(\hat{x},\hat{y}) \leq \max_{y'\in Y}f(x,y'),~\forall (x,y)\in X\times Y,$$
which implies
$$\varphi(\hat{x}) = \max_{y\in Y}f(\hat{x},y) \leq \max_{y'\in Y}f(x,y')= \varphi(x), ~\forall x \in X.$$
This means $\mathrm{Proj}_{x}\mathcal{S}_{g}\subseteq \argmin_{x\in X}\varphi(x)$. On the other hand, for any $\hat{x}\in \argmin_{x\in X}\varphi(x)$, we let $\hat{y} \in \arg\max_{y\in Y}f(\hat{x},y)$. Then it is not difficult to examine that $(\hat{x},\hat{y})$ is a global minimax point, i.e., $\argmin_{x\in X}\varphi(x) \subseteq \mathrm{Proj}_{x}\mathcal{S}_{g}$. The $\proj_x \widehat{\mathcal{S}}_{g}^{N}=\argmin_{x\in X} \hat{\varphi}_{N}(x)$ can be similarly verified. Hence (\ref{zgs2}) holds.

Then \eqref{zgs1} and \eqref{zgs2} indicates, according to \cite[Lemma 4.1]{xu2010uniform}, that
	\begin{equation}
	\label{gs4}
	\d\left( \proj_x \widehat{\mathcal{S}}_{g}^{N} , \proj_x \mathcal{S}_{g} \right) \to 0
	\end{equation}
	w.p.1 as $N\to \infty$. We know from \eqref{gs4} that $\bar{x}\in \proj_x \mathcal{S}_{g}$.
	
	Moreover, we know that
	\begin{align*}
	\abs{\widehat{\vartheta}_g^{N} - \vartheta_g} &= \abs{ \min_{x\in X}\hat{\varphi}_{N}(x) - \min_{x\in X}\varphi(x) }\\ &\leq \max_{x\in X}\abs{\hat{\varphi}_{N}(x) - \varphi(x)} \\
	&\to 0
	\end{align*}
	w.p.1 as $N\to \infty$, where $\vartheta_g$ and $\widehat{\vartheta}_g^{N}$ are optimal values of problems \eqref{MINIMAX} and \eqref{SAA_MINIMAX}, respectively. Due to Lemma \ref{Lem2} and the continuity of $f$, we know that
	\begin{align*}
	\abs{ \hat{f}_{N}(x^{N},y^{N})  -  f(\bar{x},\bar{y}) } & \leq  \abs{ \hat{f}_{N}(x^{N},y^{N})  -  f(x^{N},y^{N})} + \abs{ f(x^{N},y^{N}) - f(\bar{x},\bar{y}) }\\
	&\to 0.
	\end{align*}
	Since $\widehat{\vartheta}_g^{N} = \hat{f}_{N}(x^{N},y^{N})$, we know that $\vartheta_g = f(\bar{x},\bar{y})$, which, together with $\bar{x}\in \proj_x \mathcal{S}_{g}$, implies that $(\bar{x},\bar{y})\in  \mathcal{S}_{g}$.\qed
\end{proof}

Based on Theorem \ref{Th3}, it is well-founded for us to employ problem \eqref{SAA_MINIMAX} to approximately solve problem \eqref{MINIMAX}. In the sequel, we will focus on how to compute an $\epsilon$-first-order stationary point of problem \eqref{SAA_MINIMAX}.

\section{The QNSTR algorithm and its convergence analysis}\label{Sec3}

In this section, we propose the QNSTR algorithm to compute an $\epsilon$-first-order stationary point of problem \eqref{SAA_MINIMAX} with a fixed sample size $N$.
In the remainder of this paper, let $X=[a,b]$ and $Y=[c,d]$, where $a,b\in \mathbb{R}^n$,  $c,d\in \mathbb{R}^m$ with
$a< b$ and $c< d$ in the componentwise sense.  In this case, the projection in (\ref{gs15}) has a closed form and
the function $F_N$ can be written as

\begin{equation}\label{boxF}
F_N(z)=z - {\rm mid}(l, u, z - H_N(z)),
\end{equation}
\\
where $l,u\in \mathbb{R}^{n+m}$ with $l=(a^\top, c^\top)^\top$ and $u=(b^\top, d^\top)^\top$, ``mid'' is the middle operator in the componentwise sense, that is
$${\rm mid}(l,u, z-H_N(z))_i= \left\{\begin{array}{ll}
l_i, & {\rm if} \quad   (z - H_N(z))_i < l_i,\\
u_i, &  {\rm if} \quad  (z - H_N(z))_i > u_i,  ~ i=1,\cdots, n+m,\\
(z - H_N(z))_i,  & {\rm otherwise}.
\end{array} \right.
$$
\\
Since $X$ and $Y$ are boxes, \eqref{boxF} can be divided into two parts separably and rewritten as

\begin{equation}\label{boxF1}
    F_{N}(z)=\left(\begin{array}{l}
    F^{\mathbbm{1}}_{N}(z)\\
F^{\mathbbm{2}}_{N}(z)\end{array}\right)= \left(\begin{array}{l}
x-\mid(a,b,x-\nabla_{x}\hat{f}_{N}(x,y))\\
y-\mid(c,d,y+\nabla_{y}\hat{f}_{N}(x,y))
\end{array}\right).
\end{equation}

\subsection{Smoothing approximation}
\label{subset3}
Let $q(z)=z-H_N(z).$ The function $F_N$ is not differentiable at $z$ when $q_i(z) = l_i$ or $q_i(z) = u_i$ for some $1\le i\le n+m$. To overcome the difficulty in computation of the generalized Hessian of the nonsmooth function $F_N(z)$, we consider its smoothing approximation
	\begin{equation}
    \label{smooth}
      (F_{N,\mu})_i(z) =
	\begin{cases}
	\frac{1}{2}((H_N)_i(z) + z_i) + \frac{1}{2\mu} (u_i-q_i(z))^2 + \frac{\mu}{8} - \frac{u_i}{2}, &\text{if}~ \abs{u_i -q_i(z)}\leq \frac{\mu}{2},\\
	\frac{1}{2}((H_N)_i(z) + z_i)  - \frac{1}{2\mu} (l_i - q_i(z))^2 - \frac{\mu}{8} - \frac{l_i}{2}, &\text{if}~ \abs{l_i -q_i(z)}\leq \frac{\mu}{2},\\
	(F_N)_i(z), & \text{otherwise},
	\end{cases}
    \end{equation}
where
$0<\mu\leq \hat{\mu}:=\min_{1\leq i\leq n+m} (u_i - l_i)$.

From (\ref{boxF1}), the smoothing function $F_{N,\mu}(z)$ can also be represented as
\begin{equation}\label{boxF2}
    F_{N,\mu}(z)=\left(\begin{array}{l}
    F^{\mathbbm{1}}_{N,\mu}(z)\\
F^{\mathbbm{2}}_{N,\mu}(z)\end{array}\right),
\end{equation}
where
$F^{\mathbbm{1}}_{N,\mu}(z)$ and $F^{\mathbbm{2}}_{N,\mu}(z)$ are the smoothing approximations of $F_{N}^{\mathbbm{1}}(z)$ and $F_{N}^{\mathbbm{2}}(z)$, respectively.

We summarize some useful properties of the smoothing function $F_{N,\mu}$, which can be found in \cite{chen1998global} and \cite[Section 6]{chen1997superlinear}.

\begin{lemma}\label{Lem4}
Let $F_{N,\mu}$ be a smoothing function of $F_N$ defined in \eqref{smooth}. Then for any $\mu \in (0, \hat{\mu})$, $F_{N,\mu}$ is continuously differentiable and has the following properties.
	\begin{enumerate}
		\item[(i)]  There is a $\kappa>0$ such that for any $z\in \mathbb{R}^{m+n}$ and $\mu>0$,
$$\|F_{N,\mu}(z)-F_N(z)\|\le \kappa \mu.$$

\item[(ii)]
For any $z\in \mathbb{R}^{m+n}$, we have
$$\lim_{\mu \downarrow 0} {\rm d}(\nabla_{z} F_{N,\mu}(z), \partial_C F_N(z))=0,$$
where $\partial_CF_N(z)=\partial (F_N(z))_1 \times \partial (F_N(z))_2\times \cdots \times  \partial (F_N(z))_{n+m},$ and
$\partial (F_N(z))_i$ is the Clarke generalized gradient of $(F_N(\cdot))_i$ at $z$ for $i=1,\ldots, n+m$. Moreover, there exists a $\bar{\mu}>0$ such that for any $\mu\in (0, \bar{\mu})$, we have
$\nabla F_{N,\mu}(z)\in  \partial_C F_N(z).$
	\end{enumerate}
\end{lemma}

In Figure \ref{fig_exp} in Appendix A,  we show the approximation error $\| F_{N,\mu}(z)-F_{N}(z)\|$ over $X\times Y$ as $\mu\downarrow 0$ with different $N$.

\begin{definition}[$\epsilon$-first-order stationary points]
\label{Def_efosp}
For given $\epsilon>0$, a point $z$ is called an \emph{$\epsilon$-first-order stationary point} of problem \eqref{SAA_MINIMAX}, if $\norm{F_N(z)} \le \epsilon$.
\end{definition}
From (i) of Lemma \ref{Lem4},
if $z^*$ is an $\epsilon$-first-order stationary point of problem \eqref{SAA_MINIMAX}, i.e., $\norm{F_N(z^*)} \leq \epsilon$, then for any $\mu \in (0,  \frac{\epsilon}{\kappa})$, we have
\begin{equation}
\label{kappa}
\|F_{N,\mu}(z^*)\|-\|F_N(z^*)\|\le \|F_{N,\mu}(z^*)-F_N(z^*)\| \le \kappa \mu \le \epsilon,
\end{equation}
which implies $\|F_{N,\mu}(z^*)\|\le \|F_N(z^*)\|+ \epsilon \le 2\epsilon.$  On the other hand, if $z^*$ satisfies $\|F_{N,\mu}(z^*)\|\le \frac{\epsilon}{2}$ for some $\epsilon >0$ and $\mu\in (0, \frac{\epsilon}{2\kappa})$, then we have
$$\|F_N(z^*)\|-\|F_{N,\mu}(z^*)\|\le \|F_{N,\mu}(z^*)-F_N(z^*)\| \le \kappa \mu \le \frac{\epsilon}{2},$$
which implies $\|F_N(z^*)\|\le \|F_{N,\mu}(z^*)\|+ \frac{\epsilon}{2}\le \epsilon$, that is, $z^*$ is an $\epsilon$-first-order stationary point of problem \eqref{SAA_MINIMAX}.

Now we consider the smoothing least squares problem with a fixed small smoothing parameter $\mu>0$:
 \begin{equation}
\label{TRM_min_smooth}
    \min_{z\in \mathbb{R}^{n+m}} r_{N,\mu}(z):=\frac{1}{2}\Vert F_{N,\mu}(z )\Vert^{2}.
\end{equation}

Let $J_{N,\mu}(z)$ be the Jacobian matrix of $F_{N,\mu}(z)$. The gradient of the function $r_{N,\mu}$ is
$$\nabla r_{N,\mu}(z)=J_{N,\mu}(z)^\top F_{N,\mu}(z).$$
A vector $z^*$  is called a first-order stationary point of problem (\ref{TRM_min_smooth}) if $\nabla r_{N,\mu}(z^*)=0$.  If $J_{N,\mu}(z^*)$ is nonsingular, then $ F_{N,\mu}(z^*)=0$.  From (i) of Lemma \ref{Lem4}, $\|F_N(z^*)\|=\|F_N(z^*)-F_{N,\mu}(z^*)\|\le\kappa \mu\le \epsilon$
when $\mu\in (0, \epsilon/\kappa).$ This means that a first-order stationary point $z^*$ of problem (\ref{TRM_min_smooth}) is an $\epsilon$-first-order stationary point of problem \eqref{SAA_MINIMAX} if  $\mu\in (0, \epsilon/\kappa)$ and $J_{N,\mu}(z^*)$ is nonsingular.
Note that $\partial_CF_N(z)$  is a compact set for any $z\in X\times Y$. From (ii) of Lemma \ref{Lem4}, if  all matrices in $\partial_CF_N(z^*)$ are nonsingular, then there is $\mu_0 >0$ such that for any $\mu\in (0, \mu_0)$, $J_{N,\mu}(z^*)$ is nonsingular.

If $J_{N,\mu}(z^{*})$ is singular, the  assumptions of local convergence theorems in
\cite{Dennis,Gratton2007} for Gauss-Newton methods to solve the least squares problem (\ref{TRM_min_smooth}) fail. In the next subsection, we prove the convergence of the QNSRT algorithm  for
the least squares problem (\ref{TRM_min_smooth})
to a stationary point  of (\ref{TRM_min_smooth}) without assuming the nonsigularity of $J_{N,\mu}(z^{*})$.

\subsection{The QNSRT algorithm}

In this subsection, we present the QNSRT algorithm with a fixed sample size $N$ and a fixed smoothing parameter $\mu$.  For simplicity, in this subsection,
 we use $F(z)$, $J(z)$ and $r(z)$ to denote $F_{N,\mu}(z) $, $J_{N,\mu}(z)$ and $r_{N,\mu}(z)$, respectively. Moreover, we use $F^{\mathbbm{1}}(z)$, $F^{\mathbbm{2}}(z)$ to represent $F^{\mathbbm{1}}_{N,\mu}(z)$ and $F^{\mathbbm{2}}_{N,\mu}(z)$, respectively.

For an arbitrary point $z_0\in X\times Y$ and a positive number $R_0$, we denote  the level set  $\bar{S}:=\{z\in\mathbb{R}^{n+m}\, | \, r(z)\le r(z_0)\}$ and  define  a set $S(R_{0}):=\{z\in\mathbb{R}^{n+m} \, |\, \Vert z-z'\Vert\leq R_{0},\forall z'\in \bar{S}\}$. By the definition of $F$ and the boundedness of $X$ and $Y$, we have $r(z)=\frac{1}{2}\|F(z)\|^2 \to \infty$ if $\|z\|\to \infty$. Hence both $\bar{S}$ and $S(R_0)$ are bounded.

  Let $J_{1}(z)=\nabla F^{\mathbbm{1}}(z)$  and $J_{2}(z)=\nabla F^{\mathbbm{2}}(z)$.  Then from
$$r(z):=\frac{1}{2}\Vert F^{\mathbbm{1}}(z)\Vert^{2}+\frac{1}{2}\Vert F^{\mathbbm{2}}(z)\Vert^{2},$$
if $F^{\mathbbm{1}}$ and $F^{\mathbbm{2}}$ are twice differentiable at $z$, the Hessian matrix
$$
    \nabla^{2}r(z)=\frac{1}{2}\nabla^{2}\Vert F^{\mathbbm{1}}(z)\Vert^{2}+\frac{1}{2}\nabla^{2}\Vert F^{\mathbbm{2}}(z)\Vert^{2},
$$
can be written as
\begin{align}\label{gs12}
\nabla^2 \Vert F^{\mathbbm{1}}(z)\Vert^{2}&= J_{1}(z)^\top J_{1}(z) +\sum_{i=1}^{n} (F^{\mathbbm{1}})_i(z)\nabla^2(F^{\mathbbm{1}})_i(z),\\
\label{gs12_2}
\nabla^2 \Vert F^{\mathbbm{2}}(z)\Vert^{2}&= J_{2}(z)^\top J_{2}(z) +\sum_{i=1}^{m} (F^{\mathbbm{2}})_i(z)\nabla^2(F^{\mathbbm{2}})_i(z).
\end{align}

If $F^{\mathbbm{1}}$ and $F^{\mathbbm{2}}$ are not twice differentiable at $z$, the generalized Hessian of $r$ at $z$, denote $\partial (\nabla r(z)),$ is the convex hull of all $(m+n)\times (m+n)$ matrices obtained as the limit of a sequence of the form  $\nabla^{2}r(z^k)$, where $z^k \to z$ and $F^{\mathbbm{1}}$ and $F^{\mathbbm{2}}$ are twice differentiable at $z^k$ \cite{Clarke}. Hence from (\ref{gs12})-(\ref{gs12_2}) and the twice continuous differentiability of $\hat{f}_N$, we know that there is a positive number $M_1$ such that
$\|H\|\le M_1$ for any $H\in\partial (\nabla r(z))$,  $z\in S(R_0)$. Moreover from \cite[Proposition 2.6.5]{Clarke},
 there is a positive number $M_2$ such that
\begin{equation}\label{M-Lipscitz}
\norm{\nabla r(z) - \nabla r(z')} \leq M_2 \norm{z - z'}, \,\, \forall z, z' \in S(R_0).
\end{equation}

To give a globally convergent algorithm for problem (\ref{TRM_min_smooth}) without using the second derivatives, we keep the term $J_{1}(z)^{\top}J_{1}(z)$ and $J_{2}(z)^{\top}J_{2}(z)$ in \eqref{gs12}-\eqref{gs12_2}, and approximate
$$\sum_{i=1}^{n} (F^{\mathbbm{1}})_i(z)\nabla^2 (F^{\mathbbm{1}})_i(z) ~~ \text{and} ~~
\sum_{i=1}^{m} (F^{\mathbbm{2}})_i(z)\nabla^2(F^{\mathbbm{2}})_i(z).$$
Specifically, the Hessian matrix at the $k$-th iteration point $z_{k}$ is approximated by $H_k$ with
\begin{equation}
\label{GN_BFGS}
        H_{k}=J_{1}(z_{k})^{\top}J_{1}(z_{k})+J_{2}(z_{k})^{\top}J_{2}(z_{k})+A_{k},
\end{equation}
where

\begin{equation*}
    A_{k}=\begin{pmatrix}
            B_{k} & O \\
            O & C_{k}
        \end{pmatrix}.
\end{equation*}
Here the matrices $B_{k}$ and $C_k$ are computed by the truncated BFGS quasi-Newton formula as follows.
\begin{equation}
    \label{GN_SNFGS1}
    B_{k+1}=\begin{cases}
        \bar{B}_{k+1}& \,\ \text{if}\,\, \|\bar{B}_{k+1}\|\le \gamma \,\, \& \,\,
        \frac{(s^{\mathbbm{1}}_k)^\top v^{\mathbbm{1}}_k}{(s^{\mathbbm{1}}_k)^\top s^{\mathbbm{1}}_k }\ge \bar{\epsilon}\\
        \|F^{\mathbbm{1}}(z_{k+1}) \|I_{n}& \,\,\text{otherwise,}
    \end{cases}
    \end{equation}
    \begin{equation}
 \label{GN_SNFGS2}
    C_{k+1}=\begin{cases}
        \bar{C}_{k+1}& \,\, \text{if} \,\, \|\bar{B}_{k+1}\|\le \gamma \,\, \& \,\, \,\frac{(s^{\mathbbm{2}}_k)^\top v^{\mathbbm{2}}_k}{(s^{\mathbbm{2}}_k)^\top s^{\mathbbm{2}}_k}\ge \bar{\epsilon} \\
         \|F^{\mathbbm{2}}(z_{k+1}) \|I_m & \,\,\text{otherwise,}
    \end{cases}
\end{equation}
where
\begin{equation}
\label{GN_SBFGS11}
\bar{B}_{k+1}=
B_{k}-\frac{B_{k}s^{\mathbbm{1}}_{k}(s^{\mathbbm{1}}_{k})^{\top}B^{\top}_{k}}{(s^{\mathbbm{1}}_{k})^{\top}B_{k}s^{\mathbbm{1}}_{k}}+\frac{v^{\mathbbm{1}}_{k}(v^{\mathbbm{1}}_{k})^{\top}}{(v^{\mathbbm{1}}_{k})^{\top}s^{\mathbbm{1}}_{k}}
\end{equation}
and
\begin{equation}
\label{GN_SBFGS21}
\bar{C}_{k+1}=C_{k}-\frac{C_{k}s^{\mathbbm{2}}_{k}(s^{\mathbbm{2}}_{k})^{\top}C^{\top}_{k}}{(s^{\mathbbm{2}}_{k})^{\top}C_{k}s^{\mathbbm{2}}_{k}}+\frac{v^{\mathbbm{2}}_{k}(v^{\mathbbm{2}}_{k})^{\top}}{(v^{\mathbbm{2}}_{k})^{\top}s^{\mathbbm{2}}_{k}}.
\end{equation}
Here $\bar{\epsilon}$ and $\gamma$ are given positive parameters, 
$s^{\mathbbm{1}}_k=x_{k+1}-x_k$, $s^{\mathbbm{2}}_{k}=y_{k+1}-y_{k}$, \begin{align*}
v^{\mathbbm{1}}_{k}&=(\nabla_xF^{\mathbbm{1}}(z_{k+1})-\nabla_xF^{\mathbbm{1}}(z_{k}))^{\top}F^{\mathbbm{1}}(z_{k+1})\Vert F^{\mathbbm{1}}(z_{k+1})\Vert/\Vert F^{\mathbbm{1}}(z_{k})\Vert, \\ v^{\mathbbm{2}}_{k}&=(\nabla_yF^{\mathbbm{2}}(z_{k+1})-\nabla_yF^{\mathbbm{2}}(z_{k}))^{\top}F^{\mathbbm{2}}(z_{k+1})\Vert F^{\mathbbm{2}}(z_{k+1})\Vert/\Vert F^{\mathbbm{2}}(z_{k})\Vert.
\end{align*}

Notice that the approximation form \eqref{GN_BFGS} is proposed in \cite{zhou2010global}. However, the matrix $A_k$ in \cite{zhou2010global} is defined by using $F(z)$ and $\nabla F(z)$ at $z_{k+1}, z_k$. In this paper, we use $F^{\mathbbm{1}}(z)$, $F^{\mathbbm{2}}(z)$, $\nabla F^{\mathbbm{1}}(z)$ and $\nabla F^{\mathbbm{2}}(z)$ at $z_{k+1}, z_k$ to define a two-block diagonal matrix $A_{k}$ in \eqref{GN_BFGS}, based on the structure of VI in (\ref{boxF2}).

In \cite{zhou2010global},  a back tracking line search is
used to obtain a stationary point of the least squares problem. In this paper, we use a subspace trust-region method to solve problem \eqref{TRM_min_smooth} with global convergence guarantees. Comparing with the quasi-Newton method with back tracking line search in \cite{zhou2010global}, the QNSTR algorithm solves a strongly convex quadratic subproblem in a low dimension at  each step, which is efficient to solve large-scale min-max optimization problems with real data. See Section \ref{Sec4} for more details.

In what follows, for simplification, we use $J_{k}$ and $F_{k}$ to denote $J(z_{k})$, $F(z_{k})$, respectively.

Let $g_{k}=\nabla r(z_{k})$. Choose $\{d_k^1, \cdots, d_k^{L-1}\}$ such that
$V_{k}:=\left[\begin{matrix}
-g_{k}&
d_{k}^{1}&
\cdots&
d_{k}^{L-1}
\end{matrix}\right]\in\mathbb{R}^{(n+m)\times L}$ has $L$ linearly independent column  vectors. Let
$$c_{k}:=V_{k}^{\top}g_{k}, \quad G_{k}:=V_{k}^{\top}V_{k},\quad Q_{k}:=V_{k}^{\top}H_{k}V_{k}.$$
Then, to obtain the stepsize $\alpha_{k}$ at the iteration point $z_k$, we solve the following strongly convex quadratic program in an $L$-dimensional space:
\begin{equation}
\label{gs20}
\begin{array}{cl}
\alpha_{k}=\argmin\limits_{\alpha\in\mathbb{R}^{L}} & m_{k}(\alpha):=r(z_k)+ c_{k}^{\top}\alpha+\frac{1}{2}\alpha^{\top}Q_{k}\alpha\\
~~~~~~~~\mathrm{s.t.} & \Vert V_k\alpha\Vert \leq \Delta_k,
\end{array}
\end{equation}
where $\Delta_{k}>0$ is the trust-region radius.

A key question for solving problem \eqref{gs20} is how to compute $Q_{k}$ efficiently when $H_{k}$ is huge. In fact, $Q_k$ can be calculated efficiently without computing and storing the full information $H_k$. From \eqref{GN_BFGS}, we can write $Q_{k}$ as
\begin{equation}\label{gs11}
    Q_{k}=V^{\top}_{k}J_{k}^{\top}J_{k}V_{k} + V^{\top}_{k}A_{k}V_{k}.
\end{equation}
For the term $V_{k}^{\top}J_{k}^{\top}J_{k}V_{k}$ in \eqref{gs11}, we compute $J_{k}V_{k}$ in a columnwise way: For a sufficiently small $\epsilon>0$,
$$J_{k}g_{k}\approx \frac{F(z_{k}+\epsilon g_{k})-F(z_{k})}{\epsilon},~~ J_{k}d_{k}^{i}\approx \frac{F(z_{k}+\epsilon d^{i}_{k})-F(z_{k})}{\epsilon},~ i=1,\cdots,L-1.$$
On the other hand, $A_{k}V_{k}$ in term $V_{k}^{\top}A_{k}V_{k}$ can also be computed columnwisely by a series of vector-vector products.

We give the QNSTR algorithm in Algorithm \ref{Alg4}.

\begin{algorithm}[H]
	\caption{The QNSTR Algorithm }\label{Alg4}
	{\bf Input:} $\bar{\Delta}>0$, $\Delta_{0}\in(0,\bar{\Delta})$, $\beta_{1}<1<\beta_{2}$, $0\leq \eta < \zeta_{1} < \zeta_{2} \leq 1$, tolerance parameter $\delta>0$, $\epsilon>0$, $z_0\in \mathbb{R}^{n+m}.$
\begin{algorithmic}[1]
	
		\State If $\Vert g_k\Vert\leq\delta$ or $\Vert F(z_{k})\Vert\leq\epsilon$, terminate. Otherwise solve  \eqref{gs20} for $\alpha_{k}$.
        \State Compute the reduction ratio at iterate $k$:
          \begin{equation}
          \label{rho}
          \rho_{k}=\frac{r(z_{k})-r(z_{k}+V_{k}\alpha_{k})}{m_{k}(0)-m_{k}(\alpha_{k})}
          \end{equation}
        \If{$\rho_{k}<\zeta_{1}$}
        \State $\Delta_{k+1}=\beta_{1}\Delta_{k}$
        \Else
        \If{$\rho_{k}\ge\zeta_{2}$ and $\|V_k\alpha_k\|=\Delta_k$,}
        \State $\Delta_{k+1}=\min\{\beta_{2}\Delta_{k},\bar{\Delta}\}$
        \Else
        \State $\Delta_{k+1}=\Delta_{k}$
        \EndIf
        \EndIf
        \If{$\rho_{k} > \eta$}
        \State $z_{k+1}=z_{k}+V_k\alpha_k$
        \Else
        \State $z_{k+1}=z_{k}$
        \EndIf
	\end{algorithmic}
\end{algorithm}

Trust-region algorithms are a class of popular numerical methods for optimization problems \cite{conn2000trust,Yuan2015}.
 Our QNSTR algorithm uses the special structure of the VI in (\ref{boxF1}) and (\ref{boxF2}) to construct the subproblem \eqref{gs20}.  The global convergence of the QNSTR algorithm is given in the following theorem.

\begin{theorem}\label{th4-1}
Suppose that $X$ and $Y$ are nonempty and bounded boxes and $\hat{f}_N$ is twice continuously differentiable.  Let $\{z_{k}\}_{k=0}^{\infty}$ be generated by Algorithm \ref{Alg4}. Then there exists an $M>0$ such that $\norm{\nabla r(z) - \nabla r(z')} \leq M \norm{z - z'}$ for any $z,z'\in S(R_0)$ and $\norm{H_k} \leq M$ for $k\in \mathbb{N}$. Moreover, we have
$\lim_{k\rightarrow\infty}\|g_{k}\|=0.$
\end{theorem}

The proof is given in Appendix B.\\

To end this section, we give some remarks about Theorem \ref{th4-1}.

\begin{remark}
From Corollary 2.2.5 in  \cite{facchinei2003finite},  the continuity of $F$ and boundness of $X$ and $Y$ imply that the solution set ${\cal Z}^*$ of $F(z)=0$ is nonempty and bounded. Hence
the set of minimizers of the least squares problem $\min r(z)$ with the optimal value zero is coincident with the solution set of $F(z)=0$.  If all matrices in $\partial_C F_N(z)$ for $z\in {\cal Z}^*$ are nonsingular, we know from the compactness of $\partial_C F_N(z)$ and (ii) of Lemma \ref{Lem4} that there is a $\mu_0 >0$ such that for any $\mu\in (0, \mu_0)$, $J(z)$ is nonsingular and $\sup_{\mu\in (0,\mu_0)}\norm{J(z)^{-1}}\leq C$ for some $C>0$.
Thus,
\begin{align*}
\norm{F(z)} = \norm{(J(z)^\top)^{-1} \nabla r(z)} \leq \norm{(J(z)^\top)^{-1}}\norm{\nabla r(z)}\leq C\norm{\nabla r(z)}.
\end{align*}
We have from (i) of Lemma \ref{Lem4}, i.e., $\norm{ F_N(z) - F(z) } \leq \kappa \mu$ that
$$\norm{F_N(z)} \leq C \norm{\nabla r(z)}+ \kappa \mu.$$
Thus, for any $\epsilon>0$, we can properly select parameters $\delta$ and $\mu$ such that $\norm{\nabla r(z)}\le \delta$ and $C \delta+ \kappa \mu \leq \epsilon$.

According to Theorem \ref{th4-1}, we can find a point $z_k$ such that
$$\norm{g_k} = \norm{\nabla r(z_k)} \leq \delta,$$
with  an given stopping criterion parameter $\delta$. The numerical experiments in the next section show that the QNSTR algorithm can result in an $\epsilon$-first-order stationary point of problem \eqref{SAA_MINIMAX}.
\end{remark}

\section{Numerical experiments}\label{Sec4}

In this section, we report some numerical results via the QNSTR algorithm for finding an $\epsilon$-first-order stationary point of problem  \eqref{SAA_MINIMAX}. Also, we compare the QNSTR algorithm with some state-of-the-art algorithms for minimax problems. All of the numerical experiments in this paper are implemented on TensorFlow 1.13.1, Python 3.6.9 and Cuda 10.0 on a server with 1 Tesla P100-PCIE GPU with 16 GB memory at 1.3285 GHz and an operating system of 64 bits in the University Research Facility in Big Data Analytics (UBDA) of The Hong Kong Polytechnic University. (UBDA website: https://www.polyu.edu.hk/ubda/.)

We test our algorithm with two practical problems. One is a GAN based image generation problem for MNIST hand-writing data. The other one is a mix model for image segmentation on Digital Retinal Images for Vessel Extraction (DRIVE) data. In the first experiment, we use notation QNSTR(L) to denote that the dimension of the subspace spanned by the columns of $V_k$ is $L$ in the QNSTR algorithm and test the efficiency of QNSTR(L) under different choices of $L$ and directions $\{d^{i}_{k}\}_{i=1}^{L-1}$. In the second experiment, we apply the QNSTR algorithm to a medical image segmentation problem.

To ensure that $H_N$ is continuously differentiable, we choose Gaussian error Linear Units (GELU) \cite{hendrycks2016gaussian}
\begin{equation*}
\sigma(x) = x\int_{-\infty}^{x}\frac{e^\frac{-t^{2}}{2\Sigma^{2}}}{\sqrt{2\pi}\Sigma}dt
\end{equation*}
with $\Sigma=10^{-4}$ as the activation function in each hidden layer in $D$ and $G$, and Sigmoid
\begin{equation*}
\sigma(x)=\frac{e^{-x}}{1+e^{-x}}
\end{equation*}
as the activation function of the output layers in $D$ and $G$.

We test different choices of the subspace spanned by the columns of $V_k$. Specifically, denote
$$V^z_k= [ -g_k,   (z_{k}-z_{k-1}), \cdots,(z_{k-L+2}-z_{k-L+1})],$$
$$V_k^F=[ -g_k,  F(z_{k}), \cdots,  F(z_{k-L+2})],$$
$$ V_k^g= -[g_k, g_{k-1},  \cdots,  g_{k-L+1}]$$
for $L\ge 2.$
Among all the experiments in the sequel, the initial point $z_{0}$ is the result for running 10000 steps of the alternating Adam with step size $0.0005$. The parameters in Algorithm \ref{Alg4} are set as $\bar{\Delta}=100$, $\Delta_0=1$, $\beta_1=0.5$, $\beta_2=2$, $\eta=0.01$, $\zeta_{1}=0.02$, $\zeta_{2}=0.05$. The parameters $\bar{\epsilon}$ and $\gamma$ in \eqref{GN_SNFGS1} and \eqref{GN_SNFGS2} are chosen as $\bar{\epsilon}=10^{-4}$ and $\gamma=10^3.$

\subsection{Numerical performance on MNIST data}

In this subsection, we report some preliminary numerical results of the QNSTR algorithm for solving the following SAA counterpart
\begin{equation}
\label{SAA-GAN}
\begin{aligned}
\min_{x\in X}\max_{y\in Y}  \frac{1}{N}\sum^{N}_{i=1}\Big(\log (D(y,\xi^i_1)) +  \log (1- D(y,G(x,\xi^i_2))) \Big)
\end{aligned}
\end{equation}
of problem \eqref{GAN} by using MNIST handwritting data. We consider a two-layer GAN, where
$$\xi^i_1\in \mathbb{R}^{784}, \, \xi^i_2\in \mathbb{R}^{100}, $$
$$  W^1_G\in \mathbb{R}^{N_G^1\times 100}, \,W^2_G\in \mathbb{R}^{784\times N_G^1},\, W^1_D\in \mathbb{R}^{N^1_D\times 784},\,W^2_D\in \mathbb{R}^{1\times N^1_D}$$
with different choices of dimensions $N^1_G$ and $N^1_D$ for hidden outputs. Here $\{\xi^i_2\}$ are generated by an uniform distribution ${\cal U}$$(-1, 1)^{100}$.  All initial weight matrices $W^{i}_{G}$ and $W^{i}_{D}$ for $i=1,2$ are randomly generated by using the Gaussian distribution with mean $0$ and standard deviation $0.1$, and all initial bias vectors  $b^{i}_{G}$ and $b^{i}_{D}$ for $i=1,2$ are set to be zero. $X$ and $Y$ are set as $[-1,1]^{n}$ and $[-1,1]^{m}$, respectively.

\subsubsection{Numerical results for SAA problems}

We first study the performance of SAA problems under different sample sizes on a GAN model with a two-layer generator with $N_{G}^{1}=64$ and a two-layer discriminator with $N_{D}^{1}=64$, respectively. We set $\hat{N}=10000$ as the benchmark to approximate the original problem \eqref{GAN} and let $N=100, 500, 1000, 2000, 5000$. For each $N$, we solve $F_N(z)=0$,  50 times with different samples by using the QNSTR algorithm. We stop the iteration either
\begin{equation}
    \label{SAA_condition}
    \Vert F_{N}(z_k)\Vert\leq 10^{-5}
\end{equation}
or the number of iterations exceeds $5000$.

In these experiments, we set $\mu=10^{-8}$. We use $z^{*}_{N}$ to denote the first point that satisfies \eqref{SAA_condition} in the iteration for $N=500, 1000, 2000, 5000$, and we measure its optimality by the residual
$$\mathrm{res}_{N}:=\Vert z^{*}_{N}-\mathrm{mid}\left(l,u,z^{*}_{N}-H_{N}(z^{*}_{N}\right)\Vert.$$

Figure \ref{fig_SAA} shows the convergence of $\mathrm{res}_{N}$ to zero as $N$ grows. Table \ref{tab_SAA} presents the average of mean, standard deviation (std) and the width of 95\% CI of $\mathrm{res}_{N}$. It shows that all values decrease as the sample size $N$ increases. Both Figure \ref{fig_SAA} and Table \ref{tab_SAA} validate the convergence results in Section \ref{Sec2}.
%

\begin{figure}[htb]
	\centering
		\includegraphics[width=1\textwidth]{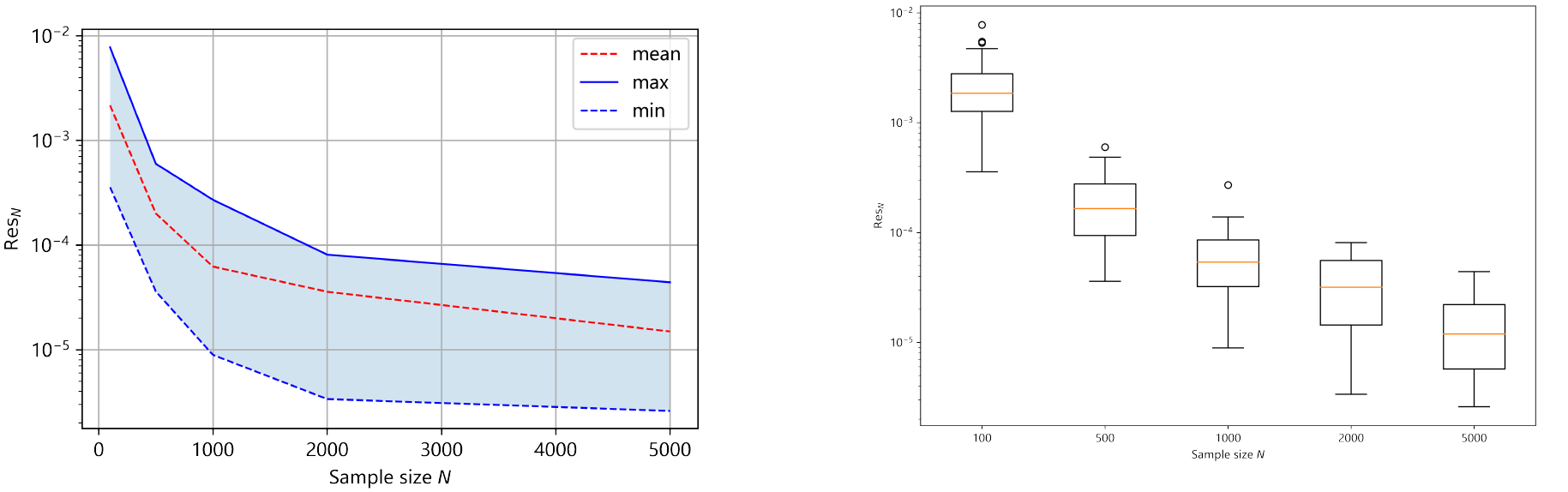}
		\caption{The convergence of $\mathrm{res}_{N}$ as $N$ grows (left: The range of $\mathrm{res}_{N}$ with different $N$, right: The boxplot of $\mathrm{res}_{N}$ with different $N$)}
	\label{fig_SAA}
\end{figure}

\begin{table}[htb]
	\centering
  \begin{footnotesize}
	\begin{tabular}{c|c|c|c|c}
        \hline
		  $N$ &   $500$ &  $1000$ & $2000$ & $5000$ \\
		\hline
		Mean &  $2.01\times10^{-4}$ & $6.22\times 10^{-5}$ & $3.58\times10^{-5}$ & $1.49\times10^{-5}$  \\
        \hline
		std &  $1.65\times10^{-4}$ & $4.95\times10^{-5}$ & $2.93\times10^{-5}$ & $1.18\times 10^{-5}$ \\
        \hline
		95\% CI & $[1.65,2.36]_{\times10^{-4}}$ & $[4.95,7.49]_{\times10^{-5}}$ & $[2.93,4.23]_{\times10^{-5}}$ & $[1.18,1.79]_{\times10^{-5}}$ \\
		\hline
	\end{tabular}
 \end{footnotesize}
	\caption{Means, variances and 95\% CIs of $\mathrm{res}_{N}$ with different $N$}
	\label{tab_SAA}
\end{table}

\subsubsection{Numerical results for smoothing approximation}
In this subsection, for a fixed sample size $N=2000$, we study how the smoothing parameter $\mu$ affects the residual $\Vert F_{N,\mu}(z)\Vert$. All numerical results in this part are based on a GAN model which is constituted of a two-layer generator with $N_{G}^{1}=64$ and a two-layer discriminator with $N_{D}^{1}=64$. Specifically, for $\mu=10^{-t}, t=1, 2, 4, 5, 6, 8$, we generate $50$ test problems, respectively. For each $\mu$, we solve problem $F_{N,\mu}(z)=0$ by the QNSTR algorithm. We stop the iteration either
condition $\|F_{N,\mu}(z_k)\|\le 10^{-5}$ holds
or the number of iterations exceeds $5000$.

We use $z^{*}_{\mu}$ to denote the first point that satisfies \eqref{SAA_condition} in the iteration, and we measure the residual of $z_{\mu}^{*}$ by
$$\mathrm{res}_{\mu}:=\Vert z^{*}_{\mu}-\mathrm{mid}\left(l,u,z^{*}_{\mu}-H_{N}(z^{*}_{\mu}\right)\Vert.$$

The numerical results are presented in Figure \ref{fig_smooth}, which shows that  $\mathrm{res}_{\mu}$ decreases as smoothing parameter $\mu$ decreases. In fact, the residual  $\mathrm{res}_{\mu}$ becomes stable when $\mu\leq 10^{-5}$. Note that it is not difficult to obtain $\kappa=\frac{\sqrt{n+m}}{8}$ in \eqref{kappa}. Also, when $N_{D}^{1}=N_{G}^{1}=64$, we have $n+m=107729$. If $\Vert F_{N,\mu}(z)\Vert \leq 10^{-5}$ and $\mu=10^{-8}$, then we have $\|F_{N}(z)\|\le \kappa \mu + \epsilon\le
\frac{\sqrt{107729}}{8}\times 10^{-8} +10^{-5}\approx1.041\times10^{-5}$. This is consistent with the theoretical results in Section \ref{Sec3}.

%

\begin{figure}[htb]
	\centering
		\includegraphics[width=1\textwidth]{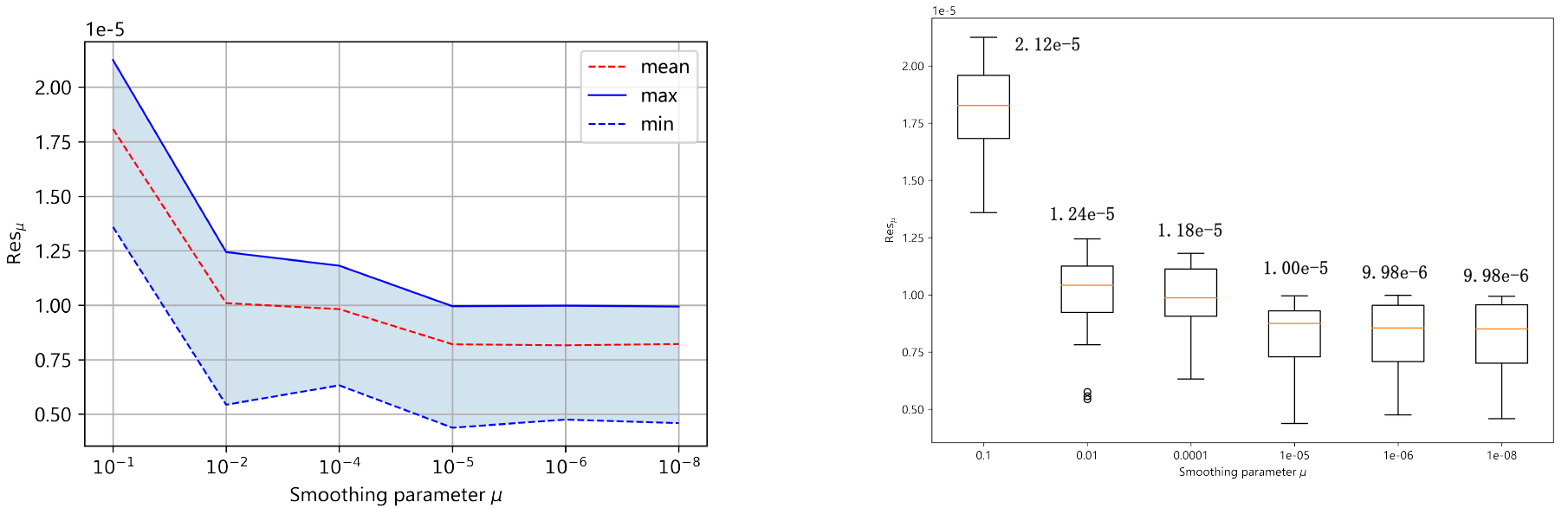}
		\caption{The convergence of $\mathrm{res}_{\mu}$ as $\mu$ decreases (left: The range of $\mathrm{res}_{\mu}$ with different $\mu$; right: The boxplot of $\mathrm{res}_{\mu}$ with different $\mu$)}
	\label{fig_smooth}
\end{figure}

\subsubsection{Comparison experiments}

In this subsection, we report some numerical results to compare the QNSTR algorithm with some commonly-used methods. To this end, we set $N=2000$ and $\mu=10^{-8}$. We use $\Vert F_{N,\mu}(z_{k})\Vert$ to measure performance of these algorithms and apply Frechet Inception Distance (FID) score to measure the quality of image generated by the generator $G$ trained by different algorithms. We terminate all these algorithms when one of the three cases holds:  $\Vert F_{N,\mu}(z_{k})\Vert \leq 10^{-5}$, $\Vert \nabla F_{N,\mu}(z_{k})^\top F_{N,\mu}(z_{k})\Vert \leq 10^{-8}$, the number of iterations exceeds $5000$.

We present the numerical results in Figures \ref{fig3}, \ref{fig4}, \ref{fig5}. It is not difficult to observe from these figures that the QNSTR algorithm outperforms.

%

\begin{figure}[htb]
	\centering
		\includegraphics[width=1\textwidth]{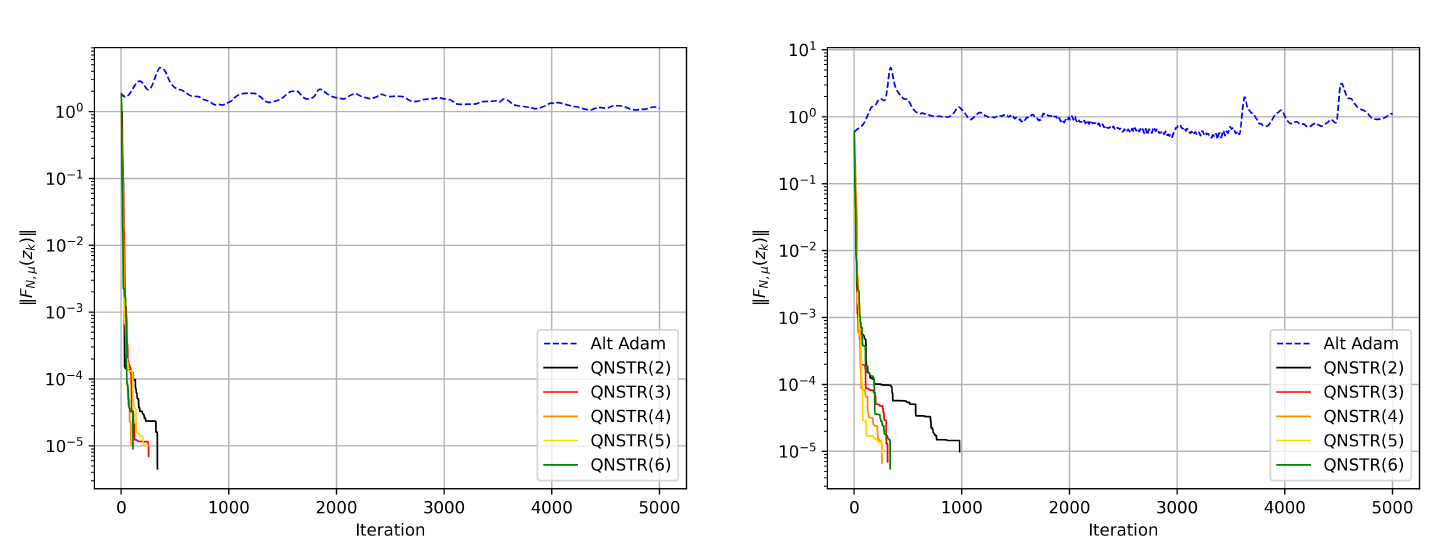}
		\caption{The residual $\Vert F_{N,\mu}(z_k)\Vert$ with $V^z_{k}$ (left: $N^{1}_{G}=N^{1}_{D}=64$; right:   $N^{1}_{G}=N^{1}_{D}=128$)}
	\label{fig3}
\end{figure}

\begin{figure}[htb]
	\centering
		\includegraphics[width=1\textwidth]{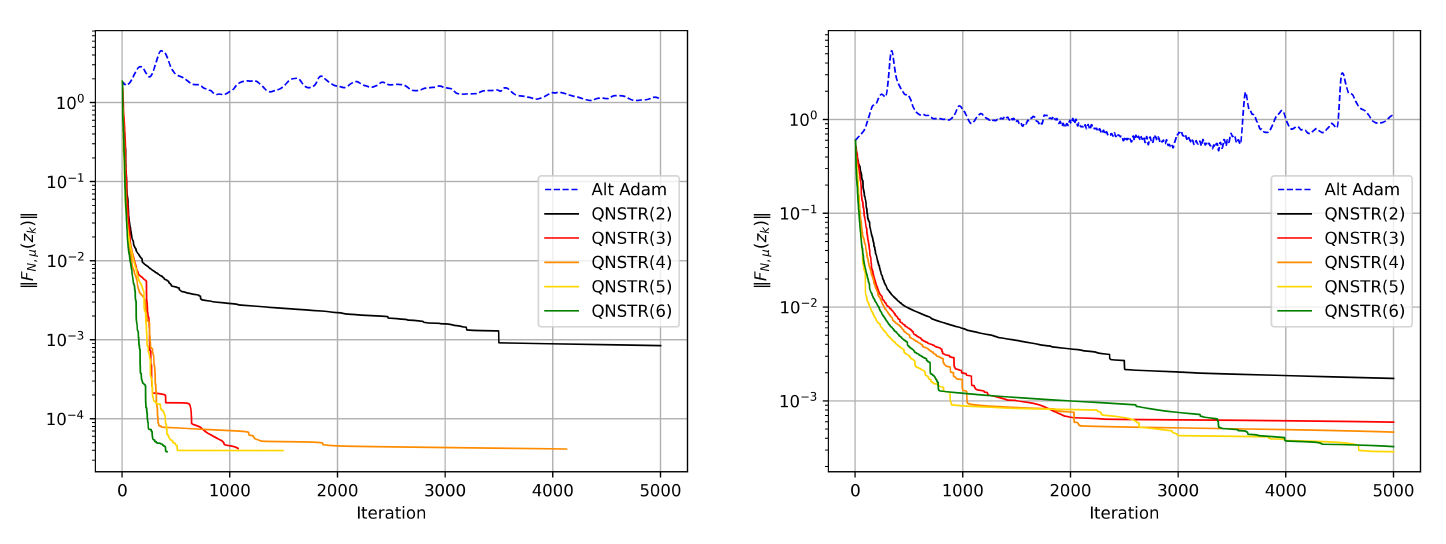}
		\caption{The residual $\Vert F_{N,\mu}(z_k)\Vert$ with $V^F_{k}$ (left: $N^{1}_{G}=N^{1}_{D}=64$; right:   $N^{1}_{G}=N^{1}_{D}=128$)}
	\label{fig4}
\end{figure}

%

\begin{figure}[htb]
	\centering
		\includegraphics[width=1\textwidth]{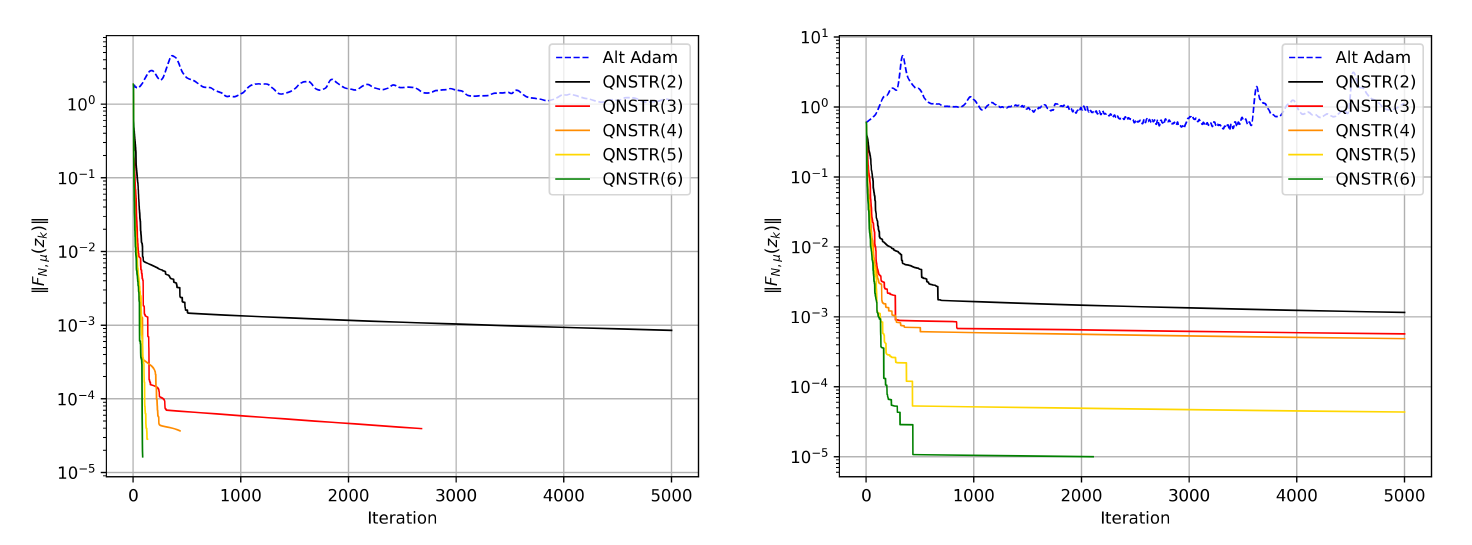}
		\caption{The residual $\Vert F_{N,\mu}(z_k)\Vert$ with $V^g_{k}$ (left: $N^{1}_{G}=N^{1}_{D}=64$; right:   $N^{1}_{G}=N^{1}_{D}=128$)}
	\label{fig5}
\end{figure}

We also compare the QNSTR algorithm with some commonly-used algorithms in training GANs including simultaneous and alternate gradient descent-ascent (GDA) \cite{zhang2022near,daskalakis2017training}, simultaneous and alternate optimistic gradient descent-ascent (OGDA) \cite{zhang2022near,daskalakis2017training}, $\gamma$-alternate Adam, projected point algorithm (PPA) \cite{liu2021first,diakonikolas2021efficient}. The initial point $z_{0}$ of all these methods are given by alternating Adam with stepsize $0.0005$ and $10000$ iterations. We use the grid search for the selection of hyper-parameters in these methods. For simultaneous and alternate GDA and simultaneous and alternate OGDA, the stepsize is $\alpha= 0.5, 0.05, 0.005, 0.001, 0.0005$. For $\gamma$-alternate Adam, the ratio is $\gamma= 1, 2, 3, 5, 10$. For PPA, the proximal parameter is $\overline{L}= 10, 100, 500, 1000, 5000$ and stepsize is $\frac{0.1}{2\overline{L}}$. Besides, the stopping criteria of the $k$-th inner loop is $\frac{0.01}{k^2}$ or the number of iterations exceeds $100$.
The comparison results between the QNSTR algorithm and the $\gamma$-alternate Adam, the QNSTR algorithm and PPA are given in Figures \ref{fig6} and \ref{fig7}, respectively. According to the results, we can see the $\gamma$-alternate Adam also shows a outstanding convergence tendency under a suitable hyper-parameter and the QNSTR algorithm can even better than the $\gamma$-alternate Adam if a suitable searching space $V_{k}$ is selected. The PPA performs poorly in these comparisons.

%

%

\begin{figure}[htb]
	\centering
		\includegraphics[width=1\textwidth]{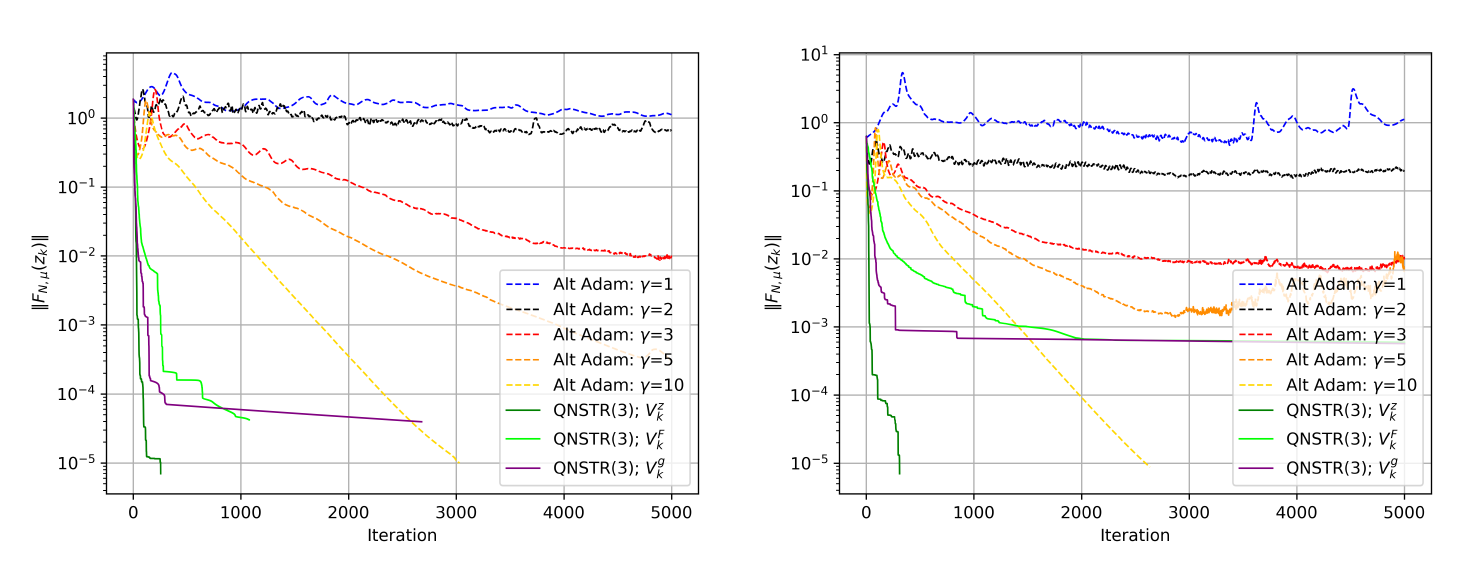}
		\caption{Comparison results between the QNSTR algorithm for $V^{z}_{k}$, $V_{k}^{F}$, $V_{k}^{g}$ and $\gamma$-alternate Adam (left: $N^{1}_{G}=N^{1}_{D}=64$; right   $N^{1}_{G}=N^{1}_{D}=128$) }
	\label{fig6}
\end{figure}

\begin{figure}[htb]
	\centering
		\includegraphics[width=1\textwidth]{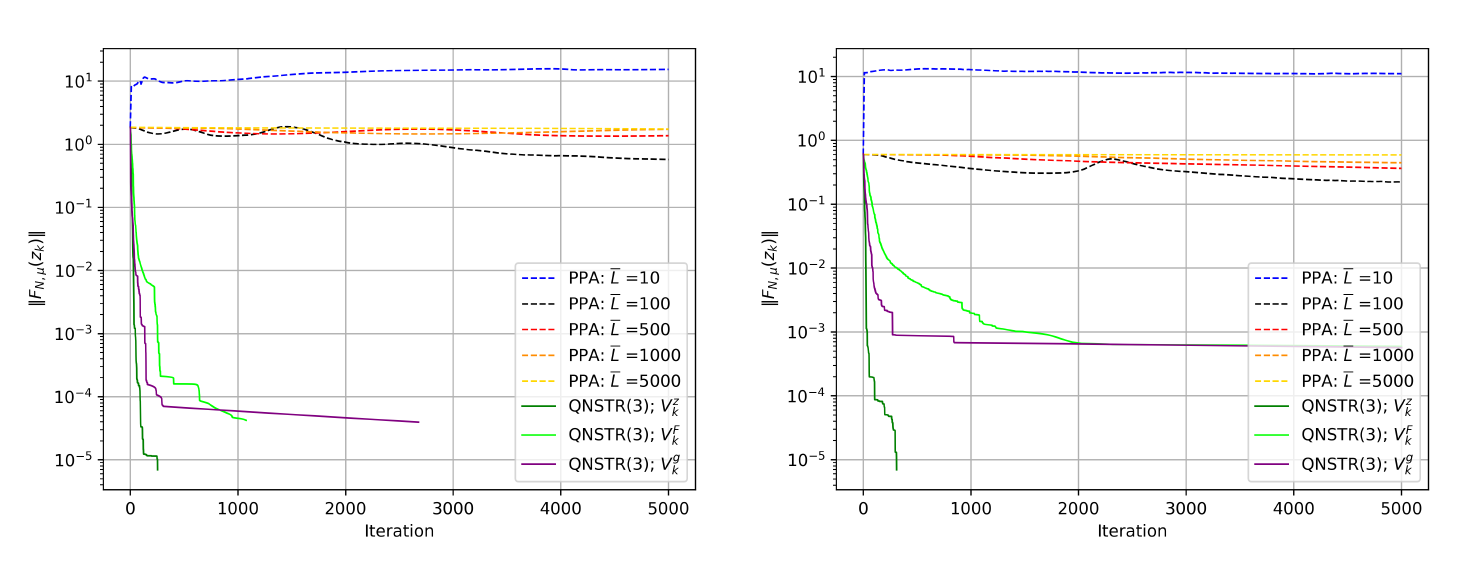}
		\caption{Comparison results between the QNSTR algorithm for $V^{z}_{k}$, $V_{k}^{F}$, $V_{k}^{g}$ and PPA (left: $N^{1}_{G}=N^{1}_{D}=64$; right   $N^{1}_{G}=N^{1}_{D}=128$) }
	\label{fig7}
\end{figure}

The comparison results between the QNSTR algorithm and simultaneous and alternate GDA, simultaneous and alternate OGDA are given in Figure \ref{fig8}. To show these comparison results more clearly, we only give the results of each method with the optimal stepsize $\alpha$ in the search range.

%
%

\begin{figure}[htb]
\centering
\includegraphics[width=1\textwidth]{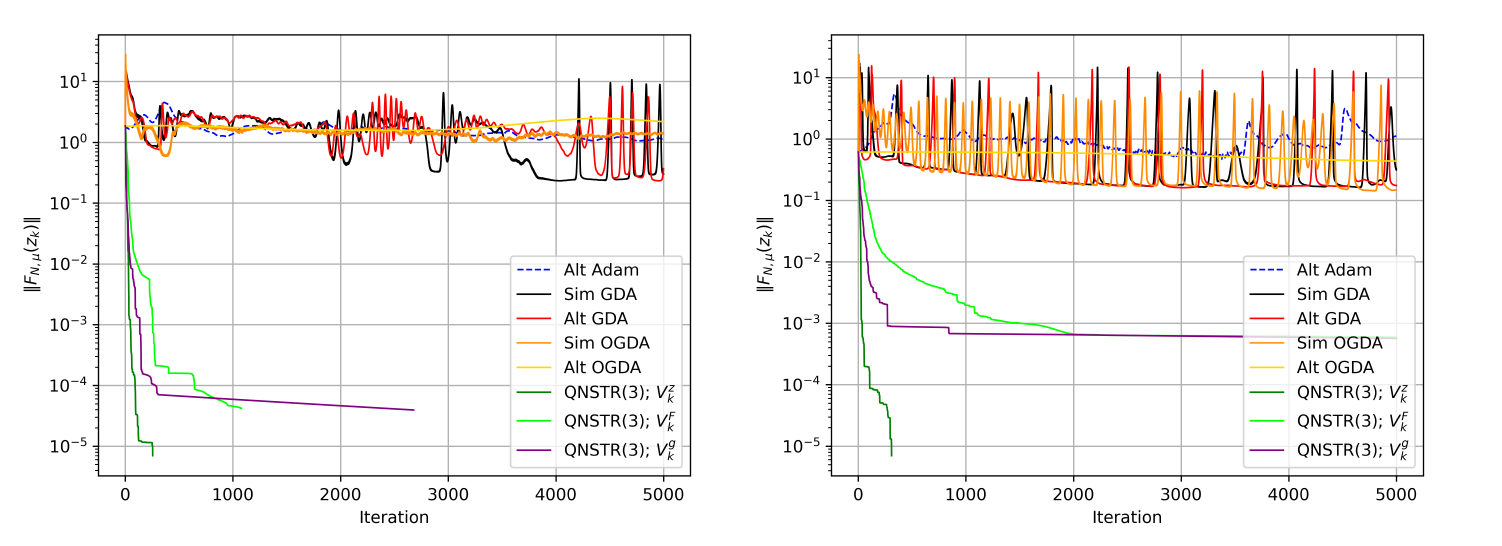}
\caption{Comparison results between the QNSTR algorithm for $V^{z}_{k}$, $V_{k}^{F}$, $V_{k}^{g}$ and simultaneous GDA, alternate GDA, simultaneous OGDA, alternate OGDA (left: $N^{1}_{G}=N^{1}_{D}=64$; right:  $N^{1}_{G}=N^{1}_{D}=128$)}
\label{fig8}
\end{figure}

We can observe from these figures that the QNSTR algorithm outperforms, which validates that the QNSTR algorithm is more efficient in finding an $\epsilon$-first-order stationary point of problem \eqref{SAA-GAN}.

We also record the final FID score of each algorithm's output. All results are given in Tables \ref{FID1} and \ref{FID2}. They show the generator of GANs trained by the QNSTR algorithm can generate high quality images.

\begin{table}[htb]

    \centering
    \begin{tabular}{c|ccccc}

        \multicolumn{6}{c}{QNSTR}  \\
          $L$   & $2$ & $3$ & $4$ & $5$ & $6$ \\
        \hline
         $V_{k}^{z}$ & 30.33 & 30.38 & 38.86 & 37.64 & 38.57  \\
         $V_{k}^{F}$ & 30.98 & 31.86 & 32.26 & 32.47 & 32.00  \\
         $V_{k}^{g}$ & 32.09 & 33.17 & 33.21 & 33.46 & 33.06  \\
         \hline
         \multicolumn{6}{c}{} \\
        \multicolumn{6}{c}{PPA}  \\
        $\overline{L}$ & $10$ & $100$ & $500$ & $1000$ & $5000$ \\
         \hline
       & 32.77 & 31.93 & 36.15 & 37.26 & 37.73 \\
         \multicolumn{6}{c}{} \\
        \multicolumn{6}{c}{ $\gamma$-alt adam }  \\
        $\gamma$ & $1$ & $2$ & $3$ & $5$ & $10$ \\
        \hline
         & 28.30 & 32.02 & 32.37 & 32.31 & 32.22  \\
         \multicolumn{6}{c}{} \\
                 \multicolumn{6}{c}{sGDA, aGDA, sOGDA, aOGDA}  \\
        $\alpha$ & $0.5$ & $0.05$ & $0.01$ & $0.005$ & $0.0005$ \\
        \hline
      sGDA   & 34.11 & 30.87 & 31.56 & 32.97 & 34.57  \\
      aGDA    & 32.21 & 30.86 & 31.53 & 31.96 & 33.87 \\
      aOGDA     & 30.73 & 30.14 & 31.47 & 32.02 & 34.06  \\
      aOGDA     & 30.97 & 31.91 & 31.87 & 32.46 & 33.82  \\
         \hline
    \end{tabular}
    \caption{FID scores of different algorithms with $N_D^1= N_G^1=64$}
    \label{FID1}
\end{table}

\begin{table}[htb]

    \centering
    \begin{tabular}{c|ccccc}

        \multicolumn{6}{c}{QNSTR}  \\
          $ L$  & $2$ & $3$ & $4$ & $5$ & $6$ \\
        \hline
         $V_{k}^{z}$ & 30.33 & 30.08 & 37.48 & 33.22 & 34.15  \\
         $V_{k}^{F}$ & 33.20 & 31.33 & 30.27 & 32.43 & 31.83  \\
         $V_{k}^{g}$ & 28.30 & 27.73 & 37.64 & 27.85 & 26.82  \\
         \hline
         \multicolumn{6}{c}{} \\
        \multicolumn{6}{c}{PPA }  \\
        $\overline{L}$ & $10$ & $ 100$ & $ 500$ & $ 1000$ & $ 5000$ \\
         \hline
        & 29.85 & 27.32 & 27.06 & 38.76 & 30.48 \\
         \multicolumn{6}{c}{} \\
        \multicolumn{6}{c}{$\gamma$-alt adam}  \\
        $\gamma$ & $1$ & $ 2$ & $ 3$ & $ 5$ & $ 10$ \\
        \hline
          & 26.13 & 27.73 & 29.64 & 30.98 & 30.70  \\
         \multicolumn{6}{c}{} \\
                 \multicolumn{6}{c}{sGDA, aGDA, sOGDA, aOGDA}  \\
        $\alpha$ & $0.5$ & $ 0.05$ & $ 0.01$ & $ 0.005$ & $ 0.0005$ \\
        \hline
         sGDA & 31.43 & 30.18 & 30.85 & 32.64 & 34.57  \\
         aGDA & 31.21 & 29.98 & 30.46 & 32.33 & 34.87 \\
         sOGDA & 30.91 & 28.91 & 30.24 & 31.85 & 34.06  \\
         aOGDA & 32.97 & 30.17 & 29.87 & 30.46 & 33.52  \\
         \hline
    \end{tabular}
    \caption{FID scores of different algorithms with $N_D^1= N_G^1=128$}
    \label{FID2}
\end{table}

\subsection{DRIVE data}\label{subsec4}

Image segmentation is an important component in many visual understanding systems, which is the process of partitioning a digital image into multiple image segments \cite{szeliski2022computer}. Image segmentation plays a central role in a broad range of applications \cite{forsyth2002computer}, including medical image analysis, autonomous vehicles (e.g., navigable surface and pedestrian detection), video surveillance and augmented reality. One of the well-known paradigm for image segmentation is based on some kinds of manual designed loss functions. However, they usually lead to the blurry segmentation boundary \cite{isola2017image}.

In 2016, Phillip et. al introduced a generative adversarial network framework into their objective function to implement an image-to-image translation problem \cite{isola2017image}, and found the blurry output given CNN under $l_{1}$ norm can be reduced. At the same year, Pauline et. al introduced a mix objective function combining by GAN and cross-entropy loss on semantic image segmentation problem \cite{luc2016semantic}, also implemented a better performance. The similar idea of  mix GAN and traditional loss can also be found in \cite{zhang2017deep,son2019towards}.

The mix model has the following form
\begin{equation}\label{SAA-MGAN}
\begin{aligned}
\min_{x\in X}\max_{y\in Y} \bigg\{ \hat{f}_{N}(x,y): = & \frac{1}{N}\sum_{i=1}^N \lambda\cdot\psi\big(\xi_1^i, G(x,\xi_2^i)\big) + \\
&\Big( \frac{1}{N}\sum_{i=1}^N \Big( \log (D(y,\xi_1^i)) + \log (1- D(y,G(x,\xi_2^i)))\Big) \bigg\},
\end{aligned}
\end{equation}
where $X$, $Y$ are two bounded boxes,  $\{(\xi^{i}_{1},\xi^{i}_{2})\}_{i=1}^{N}$ is the finite collected data, $\xi^{i}_{2}$ is the original data while the $\xi^{i}_{1}$ is the corresponding label. Problem \eqref{SAA-MGAN} is a special case of problem (\ref{SAA_MINIMAX}), which can be viewed as a discrete generative adversarial problem \eqref{SAA-GAN} with an extra classical supervision term.
The classical supervision part, i.e.,
\begin{equation}\label{DNN}
\min_{x\in X}\frac{1}{N}\sum_{i=1}^{N}\Big[ \psi\big(\xi^{i}_1, G(x,\xi^{i}_2)\big) \Big]
\end{equation}
is to minimize the difference between the output of given $\xi_{2}$ on $G$ and its corresponding label $\xi_{1}$.
The model can be regarded as a combination of a classical supervised learning problem and a generative adversarial problem with a trade-off parameter $\lambda\in \lbrack0,\infty)$. When $\lambda=0$, problem \eqref{SAA-MGAN} reduces to a classical supervised learning problem \eqref{DNN}. When $\lambda\rightarrow\infty$, problem \eqref{SAA-MGAN} tends to a vanilla GAN.

The fundoscopic exam is an important procedure to provide information to diagnose different retinal degenerative diseases such as Diabetic Retinopathy, Macular Edema and Cytomegalovirus Retinitis. A high accurate system to sketch out the blood vessel and find abnormalities on fundoscopic images is necessary. Although the supervision deep learning frameworks such as Unet are able to segment macro vessel accurately, they failed for segmenting microvessels with high certainty. In this part, we will train a Unet as a generator in our framework on DRIVE data. We download the data
includes 20 eye blood vessel images with manual segmentation label from the open
source website (https://drive.grand-challenge.org/). We applied 16 images as training data while the other 4 as testing data. In this experiment, the structure of segmentation model $G$ is U-net \cite{ronneberger2015u} which includes 18 layers with $n=121435$ parameters, and the structure of discrimination model $D$ is a deep convolutional neural network which contains 5 convolutional layers and 1 fully connected layer with $m=142625$ parameters. The feasible sets $X$ and $Y$ are set as $[-5,5]^{n}$ and $[-5,5]^{m}$, respectively.
We use activation function  GELU except Sigmoid at the output layer of $D$ and $G$. We compare our results based on problem \eqref{SAA-MGAN} with some existing models.
In our experiment, we use $\lambda=10$ and $V^z_k$ with $L=4$.

We compute traditional metrics such as F1-score, Sensitivity, Specificity and Accuracy. The form of these metrics are given as follows:
	\begin{footnotesize}	$$\mathrm{Sensivity}=\frac{1}{N}\sum_{i=1}^{N}\frac{|\mathrm{GT}_{i}\cap\mathrm{SR}_{i}|}{|\mathrm{GT}_{i}\cap\mathrm{SR}_{i}|+|\mathrm{GT}_{i}\cap\mathrm{SR}_{i}^{c}|},$$
		$$\mathrm{Specificity}=\frac{1}{N}\sum_{i=1}^{N}\frac{|\mathrm{GT}_{i}^{c}\cap\mathrm{SR}_{i}^{c}|}{|\mathrm{GT}_{i}^{c}\cap\mathrm{SR}_{i}^{c}|+|\mathrm{GT}_{i}^{c}\cap\mathrm{SR}_{i}|},$$
				$$\mathrm{Accuracy}=\frac{1}{N}\sum_{i=1}^{N}\frac{|\mathrm{GT}_{i}\cap\mathrm{SR}_{i}|+|\mathrm{GT}_{i}^{c}\cap\mathrm{SR}_{i}^{c}|}{\Omega},$$
        $$\mathrm{Precision}=\frac{1}{N}\sum_{i=1}^{N}\frac{|\mathrm{GT}_{i}\cap\mathrm{SR}_{i}|}{|\mathrm{GT}_{i}\cap\mathrm{SR}_{i}|+|\mathrm{GT}_{i}^{c}\cap\mathrm{SR}_{i}|},$$
		$$\mathrm{F1}=\frac{2\mathrm{Precision}\times\mathrm{Sensitivity}}{\mathrm{Precision}+\mathrm{Sensitivity}},$$
\end{footnotesize}
where $\Omega$ is the Universe set for all index of pixels in image, $\mathrm{GT}_{i}$ is the grouth truth vessel index for $i$-th image, $\mathrm{SR}$ is the index of pixel labelled as vessel in $i$-th image's segmentation result.
Furthermore, we compute Area Under Curve-Receiver Operating Characteristic (AUC-ROC) \cite{bradley1997use} and Structural Similarity Index Measure (SSIM) \cite{wang2004image}.

Table \ref{tab6} shows that the QNSTR algorithm for solving problem \eqref{SAA-MGAN} is more promising for blood vessel segmentation. In Figure \ref{DRIVE}, we visualize the error between vessel map generated by problem \eqref{SAA-MGAN} with the QNSTR algorithm and the manual segmentation. In Figure \ref{DRIVE1}, we compare the segmentation results of problem \eqref{SAA-MGAN} based on the QNSTR algorithm and the alternating Adam approach, respectively.

\begin{table}
	\centering
  \begin{footnotesize}
	\begin{tabular}{cccccccccc}
		Molds & F1 score & Sensitivity & Specificity & Accuracy & AUC-ROC &  SSIM \\
		\hline\\
		Residual Unet\cite{alom2018recurrent} & 0.8149 & 0.7726 & 0.9820 & 0.9553 & 0.9779 &  - \\
		RecurrentUnet\cite{alom2018recurrent} & 0.8155 & 0.7751 & 0.9816 & 0.9556 & 0.9782 &  - \\
		R2Unet\cite{alom2018recurrent} & 0.8171 & 0.7792 & 0.9813 & 0.9556 & 0.9784 & - \\
		DFUNet\cite{jin2019dunet} & 0.8190 & 0.7863 & 0.9805 & 0.9558 & 0.9779  & 0.8789 \\
		IterNet\cite{li2020iternet} & 0.8205 & 0.7735 & 0.9838 & 0.9573 & 0.9816 & 0.9008 \\
\textbf{ Alt Adam} & {0.7856} & {0.7830} & {0.9807} & {0.9551} & {0.9747} &  {0.8908} \\
		\textbf{QNSTR} & {0.7990} & {0.8327} & {0.9808} & {0.9648} & {0.9791} &  {0.8936} \\
		\hline
	\end{tabular}
 \end{footnotesize}
	\caption{The performance of QNSTR algorithm and alternating Adam for model \eqref{SAA-MGAN},  and other methods in \cite{alom2018recurrent,jin2019dunet,li2020iternet} for model (\ref{DNN})}
	\label{tab6}
\end{table}

\begin{figure}[htb]
\label{DRIVE}
	\centering
		\includegraphics[width=1\textwidth]{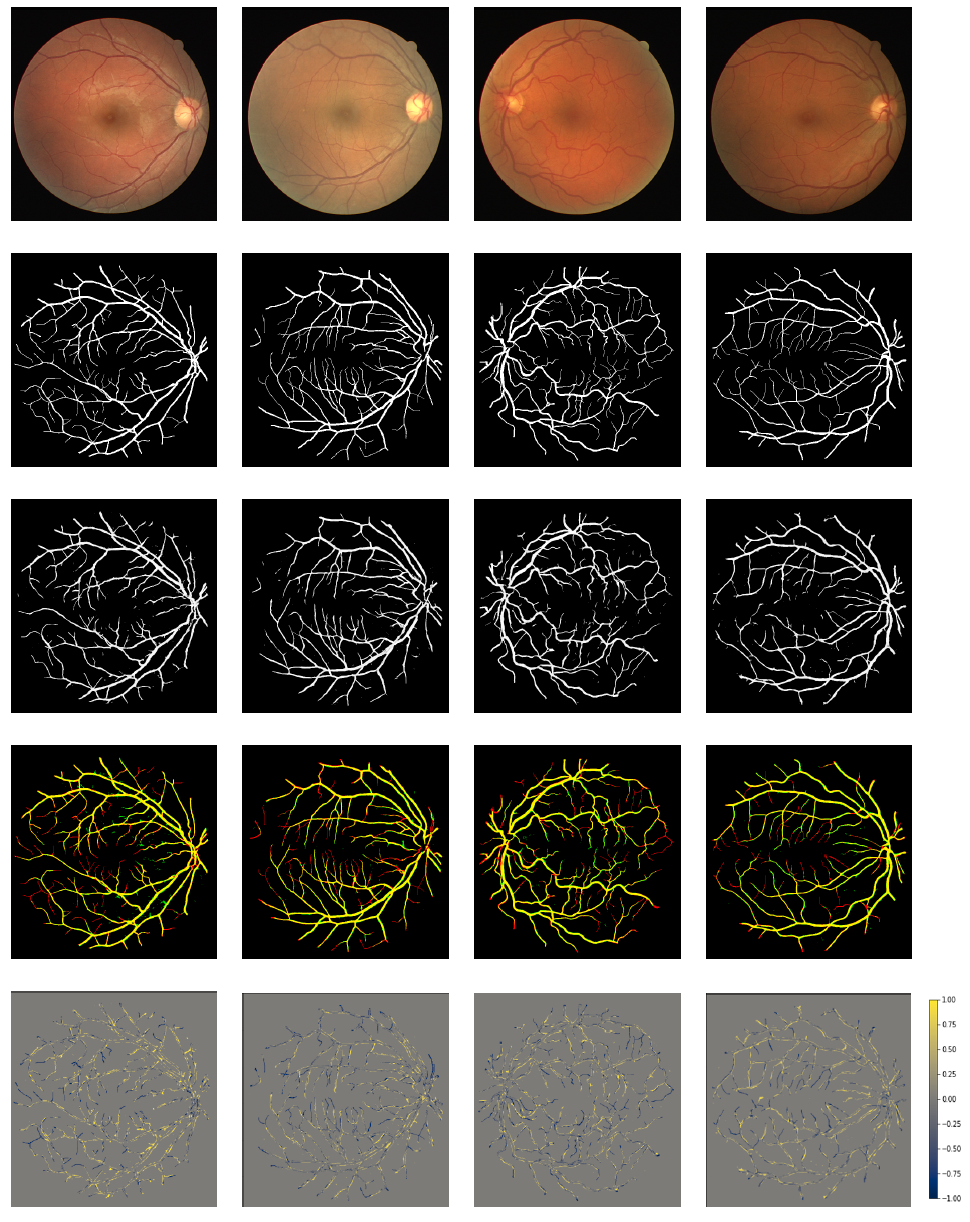}
	\caption{Row 1. fundus image, Row 2. manual segmentation, Row 3. vessel map generated by GANs with QNSTR algorithm, Row 4. yellow(correct); red(wrong); green(missing), Row 5. Error  }
	\label{DRIVE}
\end{figure}

\begin{figure}[htb]
\label{DRIVE1}
	\centering
		\includegraphics[width=1\textwidth]{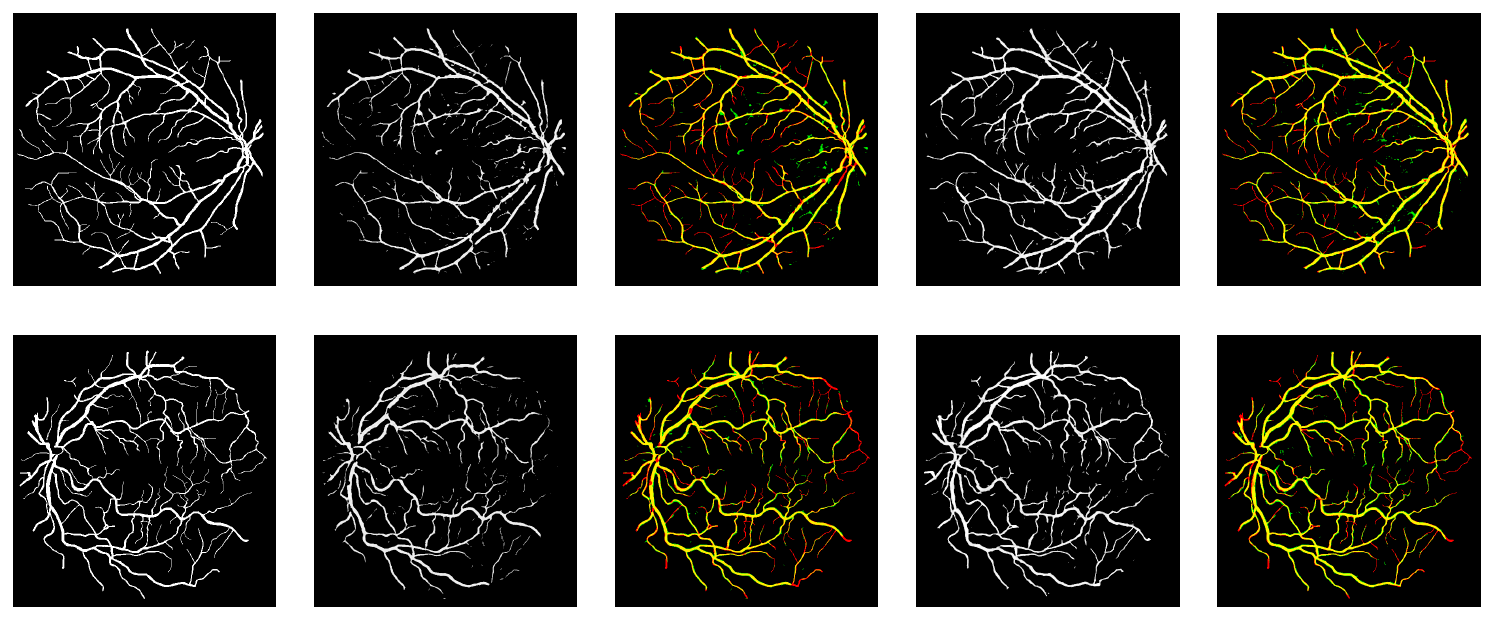}
	\caption{Comparison of Alternating Adam and QNSTR Algorithm. Columns from left to right are: 1. manual segmentation, 2. vessel map generated by GANs with Alternating Adam, 3. yellow(correct); red(wrong); green(missing) of Alternating Adam, 4. vessel map generated by GANs with proposed Algorithm, 5. yellow(correct); red(wrong); green(missing) of QNSTR Algorithm.}
	\label{DRIVE1}
\end{figure}

\section{Conclusion}\label{Sec5}

In this paper, we propose a new QNSTR algorithm for solving the large-scale min-max optimization problem (\ref{SAA_MINIMAX}) via  the nonmonotone VI (\ref{gs10}).  Based on the structure of the problem, we use a smoothing function $F(\cdot, \mu)$ to approximate the nonsmooth function $F_N$, and consider the smoothing least squares problem (\ref{TRM_min_smooth}). We adopt an adaptive quasi-Newton formula in \cite{zhou2010global} to approximate the Hessian matrix and solve a strongly convex quadratic program with ellipse constraints in a low-dimensional subspace at each step of the QNSTR algorithm.  We prove the global convergence of the QNSTR algorithm to
a stationary point of the least squares problem, which is an $\epsilon$-first-order stationary point of the min-max optimization problem if every element of the generalized Jacobian of  $F_N$ is nonsingular at the point.
In our numerical experiments, we test the QNSTR algorithm by using two real data sets: MNIST data and DRIVE data.  Preliminary numerical results validate that the QNSTR algorithm outperforms some existing algorithms.

{\bf Acknowledgment} We would like to thank Prof. Yinyu Ye for his suggestion to add more search directions in the Dimension Reduced Second-Order Method proposed in \cite{zhang2022drsom}.  This work is supported by University Research Facility in Big Data Analytics, the PolyU research
project “Analysis and applications of low information
 density signal processing”, and Hong Kong Research Grant
 Council PolyU 15300021.


\vspace{0.5in}

\section*{Appendix A: Smoothing approximation of $F_{N}(z)$}

We consider problem (\ref{SAA-GAN}) with a two-layers discriminator and a two-layers generator using MNIST handwritting data. All notations are the same as those in Section 4.1. Set $X=[-5,5]$, $Y=[-5,5]$, $N_{G}^{1}=N_{D}^{1}=64$, $N_{G}^{1}=N_{D}^{1}=128$, $\mu=10^{-t}, t=0,1,\dots,6$ and $N=1000, 2000, 10000$, respectively.
Based on the uniform distribution over $[-10,10]^{n+m}$, we generate $1000$ points  $z^i\in [-10,10]^{n+m}$, $i=1, \cdots, 1000$. Denote the ``approximation error" by
$$\text{approximation error} := \frac{1}{1000} \sum_{i=1}^{1000} \Vert F_{N,\mu}(z^i)-F_{N}(z^i)\Vert_{\infty}.$$
In Figure \ref{fig_exp}, we plot the average of ``approximation error" under different choices of $\mu$ and $N$ with 20 sets of $1000$ points in $[-10,10]^{n+m}$. From the figure, we can observe that for each $N=1000, 2000, 10000$, the ``approximation error" converges to zero as $\mu$ tends to zero.

\begin{figure}[htb]
	\centering
        \includegraphics[width=1\textwidth]{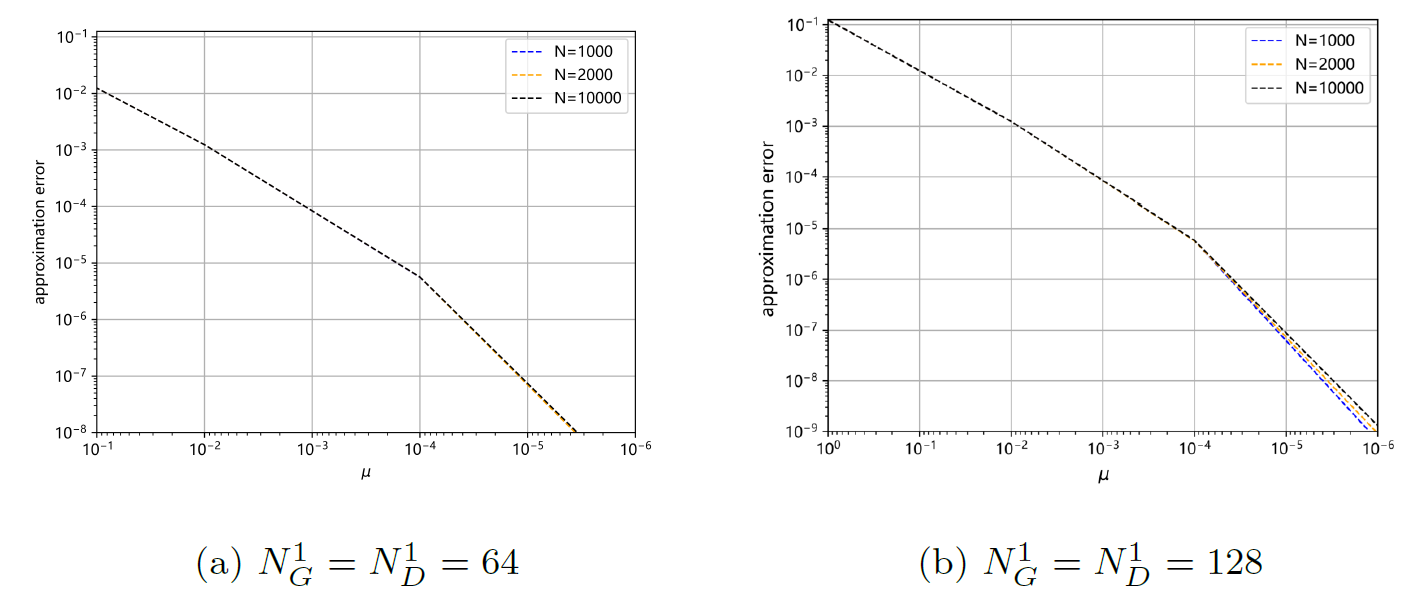}
	\caption{Smoothing approximation error of  $F_{N,\mu}(z)$ to $F_{N}(z)$ for $N=1000,2000,10000$. }
	\label{fig_exp}
\end{figure}

\section*{Appendix B: Proof of Theorem \ref{th4-1}}

As we discussed in subsection 3.2,  there is a positive number $M_1$ such that
$\|H\|\le M_1$ for any $H\in\partial (\nabla r(z))$,  $z\in S(R_0)$. Moreover
 there is a positive number $M_2$ such that $\norm{\nabla r(z) - \nabla r(z')} \leq M_2 \norm{z - z'}, \,\, \forall z, z' \in S(R_0)$.

Since $z_k$ is updated only when $\rho_k>\eta$ in Algorithm 1, we have $z_{k}\in S(R_{0})$. From the continuous differentiability  of $r$ over $S(R_{0})$, there is a positive number $M_3$ such that  $\|J_{k}\|\le M_3$ for $k \in \mathbb{N}$. Moreover, from \eqref{GN_BFGS}, $\Vert A_{k}\Vert=\max\{\Vert B_{k}\Vert, \Vert C_{k}\Vert\}\le \max\{\gamma, \|F(z_0)\|\}$.    Hence  $\norm{H_k}\leq M_{3}+\max\{\gamma, \|F(z_0)\|\}$. Let $M\ge \max \{M_2, M_3+\gamma, M_3+\|F(z_0)\|\}$.  Then we have
$\norm{\nabla r(z) - \nabla r(z')} \leq M \norm{z - z'}$ for any $z,z'\in S(R_0)$ and $\norm{H_k} \leq M$  for $k\in \mathbb{N}$.

In the rest of Appendix B,  we prove $\lim_{k\to \infty} \|g_k\|=0$.

If $g_k=0$ for some $k>0$, then Algorithm \ref{Alg4} terminates and Theorem \ref{th4-1} holds. In the remainder, we only consider the case that $g_k\neq 0$.

We next consider the following one-dimensional problem:
\begin{equation}
\label{Cauchy}
\min\limits_{\tau}  m_{k}(\tau\alpha^{s}_{k}) \quad \mathrm{s.t.} \quad \Vert\tau V_{k}\alpha^{s}_{k}\Vert\leq\Delta_{k},\quad\tau>0,
\end{equation}
where $\alpha^{s}_{k}$ is an optimal solution of
\begin{equation}
\label{gs9}
    \min_{\alpha}~ c^{\top}_{k}\alpha\quad\mathrm{s.t.}~\Vert V_{k}\alpha\Vert\leq\Delta_{k}.
\end{equation}
Let $\tau_k$ denote an arbitrary optimal solution of problem \eqref{Cauchy}. Then $\alpha^{C}_{k}:=\tau_k \alpha^{s}_{k}$ is a feasible solution of problem \eqref{gs20}.

In what follows, we give the closed form  of $\alpha^{C}_{k}$ step by step.
For this purpose, we consider the KKT condition of problem \eqref{gs9} as follows:
\begin{equation*}
\begin{split}
    \lambda G_{k}\alpha+c_{k}=0, \quad \quad &  0\le \lambda \, \bot \, \Delta_{k}^{2}-\alpha^\top G_{k}\alpha\ge 0,
    \end{split}
\end{equation*}
where $\lambda$ is a multiplier. Since $g_k\neq 0$ and $V_k$ is of full column rank, we have   $c_{k}\neq0$. Thus, $\lambda>0$, and the KKT system gives
\begin{equation*}
\alpha=-\frac{1}{\lambda}G^{-1}_{k}c_{k},\qquad \Delta_{k}^{2} =\alpha^\top G_{k}\alpha.
\end{equation*}
Then we obtain $\frac{1}{\lambda}=\sqrt{\frac{\Delta_{k}^{2}}{(G^{-1}_{k}c_{k})^{\top}G_{k}(G^{-1}_{k}c_{k})}}$,
and the solution of (\ref{gs9}) can written as
\begin{equation*}
\begin{split}
    \alpha^{s}_{k}&=-\frac{\Delta_{k}}{\sqrt{(G_{k}^{-1}c_{k})^{\top}G_{k}(G_{k}^{-1}c_{k})}}G^{-1}_{k}c_{k}\\
    &=-\frac{\Delta_{k}}{\sqrt{c_{k}^{\top}G^{-1}_{k}c_{k}}}G_{k}^{-1}c_{k}\\
    &=-\frac{\Delta_{k}}{\sqrt{g_{k}^{\top}V_{k}(V_{k}^{\top}V_{k})^{-1}V^{\top}_{k}g_{k}}}G_{k}^{-1}c_{k}\\
    &=-\frac{\Delta_{k}}{\Vert g_{k}\Vert}G_{k}^{-1}c_{k}.
\end{split}
\end{equation*}

Hence, the objective function of (\ref{Cauchy}) has the following form
\begin{equation*}
\begin{split}
~m_{k}(\tau \alpha^{s}_{k})&=r(z_k)+ \tau c_{k}^{\top}\alpha^s_k+\frac{\tau^2}{2}(\alpha_k^s)^{\top}Q_{k}\alpha^s_k\\
&= r(z_k)  -\tau\frac{\Delta_{k}}{\Vert g_{k}\Vert}c^{\top}_{k}G^{-1}_{k}c_{k}+\frac{\tau^{2}}{2}\left(\frac{\Delta_{k}}{\Vert g_{k}\Vert}\right)^{2} (G^{-1}_{k}c_{k})^{\top}Q_{k}G^{-1}_{k}c_{k}\\
    &=r(z_k)-\Delta_{k}\Vert g_{k}\Vert\tau + \frac{\tau^{2}}{2}\left(\frac{\Delta_{k}}{\Vert g_{k}\Vert}\right)^{2} g^{\top}_{k}H_{k}g_{k}
    \end{split}
\end{equation*}
and the constraint of (\ref{Cauchy}) satisfies
\begin{equation*}
\begin{split}
    \Vert\tau \alpha^{s}_{k}\Vert_{G_{k}}&=\tau \frac{\Delta_{k}}{\Vert g_{k}\Vert}\sqrt{c^{\top}_{k}G_{k}^{-1}G_{k}G^{-1}_{k}c_{k}}
    =\tau \Delta_{k}\leq\Delta_{k},
    \end{split}
\end{equation*}
which is equivalent to $0< \tau \leq 1$.

Therefore, problem \eqref{Cauchy} can be equivalently rewritten as
\begin{equation}
\label{Cauchy1}
\begin{array}{cl}
\min\limits_{\tau } & -\Delta_{k}\Vert g_{k}\Vert\tau + \frac{\tau^{2}}{2}\left(\frac{\Delta_{k}}{\Vert g_{k}\Vert}\right)^{2} g^{\top}_{k}H_{k}g_{k}, \quad \mathrm{s.t.} \quad  0< \tau \leq 1.
\end{array}
\end{equation}

Since $H_{k}$ is positive definite (see \eqref{GN_BFGS}, \eqref{GN_SBFGS11} and \eqref{GN_SBFGS21}), problem \eqref{Cauchy1} has the unique solution
\begin{equation*}
    \tau_{k}=\min\left(\Vert g_{k}\Vert^{3}/(\Delta_{k}g^{\top}_{k}H_{k}g_{k}),1\right).
\end{equation*}
Finally, we obtain
\begin{equation}
\label{Cauchy_step}
    \alpha^{C}_{k} = -\min\left(\Vert g_{k}\Vert^{3}/(\Delta_{k}g^{\top}_{k}H_{k}g_{k}),1\right)\frac{\Delta_{k}}{\Vert g_{k}\Vert}G_{k}^{-1}c_{k}.
\end{equation}

\begin{lemma}\label{Lem3}
Let $\alpha_{k}$ be the unique optimal solution of subproblem \eqref{gs20} in the $k$-th step.
Then
\begin{equation*}
    m_{k}(0)-m_{k}(\alpha_{k})\geq\frac{1}{2}\Vert g_{k}\Vert\min\left(\Delta_{k}, \frac{\Vert g_{k}\Vert}{\Vert H_{k}\Vert}\right).
\end{equation*}
\end{lemma}

\begin{proof}
Since $\alpha_{k}^{C}$ is a feasible solution of problem \eqref{gs20}, we have
\begin{equation*}
    m_{k}(0)-m_{k}(\alpha_{k})\geq m_{k}(0)-m_{k}(\alpha_{k}^{C}).
\end{equation*}
In what follows, we verify
    \begin{equation*}
        m_{k}(0)-m_{k}(\alpha^{C}_{k})\geq\frac{1}{2}\Vert g_{k}\Vert\min\left(\Delta_{k}, \frac{\Vert g_{k}\Vert}{\Vert H_{k}\Vert}\right).
    \end{equation*}

If $\Vert g_{k}\Vert^{3}/(\Delta_{k}g^{\top}_{k}H_{k}g_{k})<1$, substituting $\alpha_{k}^{C}$ (see \eqref{Cauchy_step}) into \eqref{gs20}, we have
\begin{equation}
\label{case1}
    \begin{split}
        m_{k}(0)-m_{k}(\alpha^{C}_{k}) &= -c_{k}^{\top}\alpha^{C}_{k}-\frac{1}{2}(\alpha^{C}_{k})^{\top}Q_{k}\alpha^{C}_{k}\\
        &= \frac{\Vert g_{k}\Vert^{2}}{g^{\top}_{k}H_{k}g_{k}}c^{\top}_{k}G^{-1}_{k}c_{k}-\frac{1}{2}\frac{\Vert g_{k}\Vert^{4}}{(g^{\top}_{k}H_{k}g_{k})^{2}}c^{\top}_{k}G^{-1}_{k}Q_{k}G_{k}^{-1}c_{k}\\
        &= \frac{\Vert g_{k}\Vert^{4}}{ g^{\top}_{k}H_{k}g_{k} }-\frac{1}{2}\frac{\Vert g_{k}\Vert^{4}}{(g^{\top}_{k}H_{k}g_{k})^{2}}g_{k}^{\top}H_{k}g_{k}\\
        &= \frac{1}{2}\frac{\Vert g_{k}\Vert^{4}}{g^{\top}_{k}H_{k}g_{k}}  \geq \frac{1}{2}\frac{\Vert g_{k}\Vert^{4}}{\Vert H_{k}\Vert\Vert g_{k}\Vert^{2}}
        = \frac{1}{2} \frac{\Vert g_{k}\Vert^2}{\Vert H_{k}\Vert}.
    \end{split}
\end{equation}

If $\Vert g_{k}\Vert^{3}/(\Delta_{k}g^{\top}_{k}H_{k}g_{k})\geq1$ (i.e., $g^{\top}_{k}H_{k}g_{k}\leq\frac{\Vert g_{k}\Vert^{3}}{\Delta_{k}}$), we have
\begin{equation}
\label{case2}
    \begin{split}
        m_{k}(0)-m_{k}(\alpha^{C}_{k})&=-c_{k}^{\top}\alpha^{C}_{k}-\frac{1}{2}(\alpha^{C}_{k})^{\top}Q_{k}\alpha^{C}_{k}\\
        &=\frac{\Delta_{k}}{\Vert g_{k}\Vert}\Vert c^{\top}_{k}G^{-1}_{k}c_{k}\|-\frac{1}{2}\frac{\Delta^{2}_{k}}{\Vert g_{k}\Vert^{2}}c^{\top}_{k}G_{k}^{-1}Q_{k}G_{k}^{-1}c_{k}\\
        &=\Delta_{k}\Vert g_{k}\Vert -\frac{1}{2}\frac{\Delta_{k}^{2}}{\Vert g_{k}\Vert^{2}}g^{\top}_{k}H_{k}g_{k}\\
        &\geq\Delta_{k}\Vert g_{k}\Vert-\frac{1}{2}\frac{\Delta_{k}^{2}}{\Vert g_{k}\Vert^{2}}\frac{\Vert g_{k}\Vert^{3}}{\Delta_{k}}\\
        &=\frac{1}{2} \Vert g_{k}\Vert \Delta_{k}.
    \end{split}
\end{equation}
Combining \eqref{case1} and \eqref{case2}, we complete the proof.\qed
\end{proof}

\begin{lemma}
\label{lemma_c}
Under assumptions of Theorem \ref{th4-1},
   for any index $k$, there exists a $\bar{k}>k$ such that $\Vert g_{\bar{k}}\Vert<\Vert g_{k}\Vert/2$.
\end{lemma}
\begin{proof}
    We give the proof by contradiction.
    Suppose that there exists a $\hat{k}$ with $\Vert g_{\hat{k}}\Vert=2\epsilon$ for some $\epsilon>0$, and $\Vert g_{k}\Vert{\geq}\epsilon$, $\forall k\geq\hat{k}$. Then we know from Lemma \ref{Lem3} that
    \begin{equation}
    \label{Cauchy_eps}
    m_{k}(0)-m_{k}(\alpha_{k})\geq\frac{1}{2}\Vert g_{k}\Vert\min\left(\Delta_{k},\frac{\Vert g_{k}\Vert}{\Vert H_{k}\Vert}\right)\geq \frac{1}{2}\epsilon\min\left(\Delta_{k},\frac{\epsilon}{M}\right).
    \end{equation}

According to the definition of $\rho_{k}$ in \eqref{rho}, we have
\begin{equation}
\label{gs2}
    \begin{split}
        |\rho_{k}-1|&=\left|\frac{r(z_{k})-r(z_{k}+V_{k}\alpha_{k})-(m_{k}(0)-m_{k}(\alpha_{k}))}{m_{k}(0)-m_{k}(\alpha_{k})}\right|\\
        &=\left|\frac{m_{k}(\alpha_{k})-r(z_{k}+V_{k}\alpha_{k})}{m_{k}(0)-m_{k}(\alpha_{k})}\right|.
    \end{split}
\end{equation}

By Taylor expansion, we have
\begin{equation*}
r(z_{k}+V_{k}\alpha_{k})=r(z_{k})+g_k^{\top}V_{k}\alpha_{k}+\int_{0}^{1}(\nabla r(z_{k}+tV_{k}\alpha_{k})- \nabla r(z_{k}))^{\top}V_{k}\alpha_{k}\d t.
\end{equation*}
Then
\begin{equation}
\label{gs5}
\begin{split}
    |m_{k}(\alpha_{k})-r(z_{k}+V_{k}\alpha_{k})|&=\left|\frac{1}{2}\alpha_{k}^{\top}Q_{k}\alpha_{k}-\int_{0}^{1}(\nabla r(z_{k}+tV_{k}\alpha_{k})- \nabla r(z_{k}))^{\top}V_{k}\alpha_{k}\d t \right|\\
    &\leq(M /2)\Vert V_{k}\alpha_{k}\Vert^{2}+M \Vert V_{k}\alpha_{k} \Vert^{2}\leq 3\Delta_{k}^{2}M/2,
    \end{split}
\end{equation}
where the first inequality follows from $Q_k=V_{k}^{\top}H_{k}V_{k}$ and the mean-value theorem, and the second inequality follows from $\Vert V_{k}\alpha_{k} \Vert \leq \Delta_{k}$ due to the constraint in problem \eqref{gs20}.

Then, by \eqref{Cauchy_eps}, \eqref{gs2} and  \eqref{gs5}, we get
\begin{equation*}
    |\rho_{k}-1|\leq\frac{3\Delta_{k}^{2}M/2}{(\epsilon/2)\min(\Delta_{k},\epsilon/M)}.
\end{equation*}

Denote
\begin{equation*}
    \tilde{\Delta}:=\min\left(\frac{(1-\zeta_{1})\epsilon}{3M}, R_{0}\right).
\end{equation*}
For any $\Delta_{k}\leq\tilde{\Delta}$,
we have
$$
\begin{aligned}
|\rho_{k}-1| &\leq\frac{3\Delta_{k}^{2}M/2}{(\epsilon/2)\min(\Delta_{k},\epsilon/M)} =\frac{3M \Delta_{k}^{2}}{\epsilon\Delta_{k}}
=\frac{3M\Delta_{k}}{\epsilon}\leq\frac{3M\tilde{\Delta}}{\epsilon}\leq 1-\zeta_{1},
\end{aligned}
$$
which implies $\rho_{k}\geq \zeta_{1}$, where the first equality follows from the fact that
$$\Delta_{k}\leq\tilde{\Delta}=\min\left(\frac{(1-\zeta_{1})\epsilon}{3M},R_{0}\right)\leq \frac{\epsilon}{3M}<\frac{\epsilon}{M}.$$
The above observation together with update rules in Algorithm \ref{Alg4} indicates that $\Delta_{k+1}\geq\Delta_{k}$ when $\Delta_{k}\leq\tilde{\Delta}$ (and thus $\rho_{k}\geq \zeta_{1}$). In other words, if $\Delta_{k}>\tilde{\Delta}$, $\rho_{k}<\zeta_{1}$ holds. In this case,  $\Delta_{k+1}=\beta_{1}\Delta_{k}>\beta_{1}\tilde{\Delta}$. To summarize the two cases, we then have
\begin{equation}
\label{contradiction}
    \Delta_{k}\geq\min(\Delta_{k-1},\beta_{1}\tilde{\Delta})\geq\cdots\geq\min(\Delta_{\hat{k}},\beta_{1}\tilde{\Delta}),\quad\forall k\geq \hat{k}.
\end{equation}

Then we proved sequence $\{\Delta_{k}\}_{k\geq\hat{k}}^{\infty}$ is bounded from below. Note that there exists an infinite subsequence, denoted by $\mathcal{K}$, of $\{\hat{k}, \hat{k}+1, \cdots\}$ such that, for any $k\in\mathcal{K}$, one of the following two cases holds.

Case 1: $\Delta_{k+1}=\beta_{1}\Delta_{k}$.  It is easy to obtain $\Delta_{k}\rightarrow 0$ as $k\overset{\mathcal{K}}{\rightarrow}\infty$ since $\beta_{1}<1$, which is contradicted to the fact that $\Delta_{k}$ is bounded from below (see \eqref{contradiction}).

Case 2: $\rho_{k}\geq\zeta_{1}$. We have from the definition of $\rho_k$ (see \eqref{rho}) and $\rho_{k}\geq\zeta_{1}$ that
\begin{equation*}
    r(z_{k})-r(z_{k+1})\geq\zeta_{1}(m_{k}(0)-m_{k}(\alpha_{k}))\geq\zeta_{1}\frac{1}{2}\epsilon\min(\Delta_{k},\epsilon/M)>0,
\end{equation*}
where the second inequality follows from Lemma \ref{Lem3} and $\Vert g_{k}\Vert\geq\epsilon$ for $k\in\mathcal{K}$.

Therefore, $\{r(z_k)\}_{k\in\mathcal{K}}$ is strictly decreasing.
Since $\{r(z_k)\}_{k\in\mathcal{K}}$ is bounded from below (note that $r(z) \geq 0$ for any $z$), we know that the sequence $\{r(z_k)\}_{k\in\mathcal{K}}$ is convergent and $r(z_{k})-r(z_{k+1})\downarrow0$ as $k\overset{\mathcal{K}}{\rightarrow}\infty$. Thus, $\Delta_{k} \to 0$ as $k\overset{\mathcal{K}}{\rightarrow}\infty$,
which is also contradicted with \eqref{contradiction}.\qed
\end{proof}

Now we are  ready to prove $\lim_{k\to \infty} \|g_k\|=0$.

\begin{proof} [of  $\, \lim_{k\to \infty} \|g_k\|=0$]  Let
\begin{equation}
\label{gs8}
    \epsilon:=\frac{1}{2}\Vert g_{k}\Vert \quad \text{and} \quad  R:=\min\left(\frac{\epsilon}{M}, R_{0}\right).
\end{equation}
Note that $\mathcal{B}(z_{k},R)=\{z: \Vert z-z_{k}\Vert\leq R\} \subseteq S(R_{0})$, and thus $\nabla r(\cdot)$ is Lipschitz continuous on $\mathcal{B}(z_{k},R)$ with Lipschitz modulus $M$. Thus, for $\forall z\in\mathcal{B}(z_{k},R)$, we have
$$\Vert \nabla r(z)-\nabla r(z_{k})\Vert \leq M \norm{z -z_k} \leq M R = M\min\left(\frac{\epsilon}{M}, R_{0}\right) \leq \epsilon.$$
For $\forall z\in\mathcal{B}(z_{k},R)$, we have by the triangle inequality that
\begin{align*}
    \Vert \nabla r(z)\Vert & \geq \Vert g_k \Vert - \Vert \nabla r(z)-\nabla r(z_{k})\Vert
    = 2\epsilon - \Vert \nabla r(z)-\nabla r(z_{k})\Vert
    \geq2\epsilon-\epsilon
    =\epsilon.
\end{align*}
According to Lemma \ref{lemma_c}, we know that there exists an index $l\geq k$ satisfying $\Vert g_{l+1}\Vert<\epsilon$. Moreover, we assume that $z_{l+1}$ is the first point that iterates out of the ball $\mathcal{B}(z_{k},R)$ after $z_{k}$ as well as satisfying $\Vert g_{l+1}\Vert<\epsilon$.
Thus, $\Vert g_{i}\Vert\geq\epsilon$ for $i=k,k+1,\cdots,l$. Then we have
\begin{equation}
\label{gs6}
\begin{split}
    r(z_{k})-r(z_{l+1})&=\sum_{i=k}^{l}r(z_{i})-r(z_{i+1})
    = \sum_{i=k,\atop z_{i}\neq z_{i+1}}^{l}\rho_i (m_{i}(0)-m_{i}(\alpha_{i}))\\
    &\geq\sum_{i=k,\atop z_{i}\neq z_{i+1}}^{l}\eta(m_{i}(0)-m_{i}(\alpha_{i}))
    \geq \frac{\eta}{2} \epsilon \sum_{i=k,\atop z_{i}\neq z_{i+1}}^{l}\min\left(\Delta_{i},\frac{\epsilon}{M}\right),
    \end{split}
\end{equation}
where the second equality follows from \eqref{rho}, the first inequality follows from $\rho_{i}<\eta$ when $z_{i}\neq z_{i+1}$.  Since $\norm{g_k} = 2\epsilon$ and $\norm{g_{l+1}} <\epsilon$, we have $z_{l+1}\neq z_{k}$, which implies that $\{k,\cdots,l\} \cap \{j: z_{j}\neq z_{j+1}\} \neq \emptyset$.

If $\Delta_{i}\leq\epsilon/M$ for all $i\in \{k,\cdots,l\} \cap \{j: z_{j}\neq z_{j+1}\}$, we continue \eqref{gs6} as follows:
\begin{align*}
r(z_{k})-r(z_{l+1})&\geq \frac{\eta}{2}\epsilon\sum_{i=k,\atop z_{i}\neq z_{i+1}}^{l}\Delta_{i}
\geq \frac{\eta}{2}\epsilon\sum_{i=k}^{l}\norm{z_{i+1} - z_i}\\
&\geq \frac{\eta}{2}\epsilon \norm{z_k - z_{l+1}}
\geq\frac{\eta}{2}\epsilon R=\frac{\eta}{2}\epsilon\min\left(\frac{\epsilon}{M}, R_{0}\right),
\end{align*}
where the second inequality follows from $\norm{z_{i+1} - z_i} \leq \Delta_i$, the third inequality follows from the triangle inequality, the last inequality follows from the fact that $z_{l+1}$ is the first point that iterates out of the ball $\mathcal{B}(z_{k},R)$ after $z_{k}$.

If $\Delta_{i}>\epsilon/M$ for some $i\in \{k,\cdots,l\} \cap \{j: z_{j}\neq z_{j+1}\}$, we continue \eqref{gs6} as follows:
\begin{equation*}
    r(z_{k})-r(z_{l+1})\geq \frac{\eta }{2}\epsilon \sum_{i=k,\atop z_{i}\neq z_{i+1}}^{l} \frac{\epsilon}{M} \geq \frac{\eta }{2}\epsilon \frac{\epsilon}{M},
\end{equation*}
where the last inequality follows from $\{k,\cdots,l\} \cap \{j: z_{j}\neq z_{j+1}\} \neq \emptyset$. To summarize, we obtain
\begin{equation}
\label{gs7}
r(z_{k})-r(z_{l+1}) \geq \frac{\eta}{2}\epsilon\min\left(\frac{\epsilon}{M},R_{0}\right).
\end{equation}

Since the sequence $\{r(z_{i})\}_{i=0}^{\infty}$ is a decreasing and bounded sequence from below, there exists $r^*\geq 0$ such that
$    \lim_{i\rightarrow\infty}r(z_{i})=r^{*}.$ Hence
$$
\begin{aligned}
r(z_{k})-r^{*}&\geq r(z_{k})-r(z_{l+1})\geq\frac{\eta}{2}\epsilon\min\left(\frac{\epsilon}{M},R_{0}\right)
=\frac{\eta}{4}\Vert g_{k}\Vert\min\left(\frac{\Vert g_{k}\Vert}{2M}, R_{0}\right),
\end{aligned}
$$
where the second inequality follows from \eqref{gs7}, the last equality follows from \eqref{gs8}.

Due to the arbitrariness of $k$, by letting $k\rightarrow\infty$, we know
$$\frac{\eta}{4}\Vert g_{k}\Vert\min\left(\frac{\Vert g_{k}\Vert}{2M}, R_{0}\right) \to 0,$$ which implies $\lim_{k\rightarrow\infty}\Vert g_{k}\Vert=0$. Then the proof is complete.\qed
\end{proof}

\end{document}